\documentclass{amsart}
\usepackage{amsmath,amssymb, amsbsy}
\usepackage[dvips]{graphicx}
\usepackage{color}
\usepackage[latin1]{inputenc}
\usepackage[active]{srcltx}
\usepackage{wrapfig}
\usepackage{graphicx}
\usepackage[usenames,dvipsnames]{pstricks}
\usepackage{epsfig}
 \usepackage{pst-grad} % For gradients
 \usepackage{pst-plot} % For axes
\usepackage{epic}

\renewcommand{\O}{\Omega }
\newcommand{\g}{\gamma }
\newcommand{\R}{{\mathbb R}}
\newcommand{\ren}{\R^{N}}
\newcommand{\N}{{\mathbb N}}

\renewcommand{\l }{\lambda }

\renewcommand{\O }{\Omega }

\newcommand{\iy}{\infty}

\newcommand{\inn}{\text{  in   }}
\newcommand{\dyle}{\displaystyle}
\newcommand{\dint}{\dyle\int}

\newcommand{\e }{\varepsilon}

\def\lip{ \mbox{ Lip }}

\renewcommand{\ge }{\geqslant}
\renewcommand{\geq }{\geqslant}
\renewcommand{\le }{\leqslant}
\renewcommand{\leq }{\leqslant}

\newenvironment{pfn}[1]{\noindent{\it Proof of
    {#1}.\enspace}}{\hfill\qed\medskip}\newtheorem{Theorem}{Theorem}[section]

\newtheorem{Lemma}[Theorem]{Lemma}
\newtheorem{Proposition}[Theorem]{Proposition}
\theoremstyle{definition}
\newtheorem{Definition}[Theorem]{Definition}

\newtheorem{remark}[Theorem]{Remark}

\begin{document}

\title[Effect of the Hardy potential on the summability]
{The effect of the Hardy potential in some Calderón-Zygmund properties for the fractional Laplacian}
\thanks{Work partially supported by Project MTM2013--40846-P, MINECO, Spain}
\author[B. Abdellaoui, M. Medina, I. Peral, A. Primo ]{Boumediene abdellaoui, Mar\'{i}a Medina, Ireneo Peral, Ana Primo }
\address{\hbox{\parbox{5.7in}{\medskip\noindent {$*$ Laboratoire d'Analyse Nonlin\'eaire et Math\'ematiques
Appliqu\'ees. \hfill \break\indent D\'epartement de
Math\'ematiques, Universit\'e Abou Bakr Belka\"{\i}d, Tlemcen,
\hfill\break\indent Tlemcen 13000, Algeria.}}}}
\address{\hbox{\parbox{5.7in}{\medskip\noindent{Departamento de Matem\'aticas,\\ Universidad Aut\'onoma de Madrid,\\
        28049, Madrid, Spain. \\[3pt]
        \em{E-mail addresses: }{\tt boumediene.abdellaoui@uam.es, \tt ireneo.peral@uam.es, maria.medina@uam.es, ana.primo@uam.es
         }.}}}}
\date{\today}
\begin{abstract}
The goal of this paper is to study the effect of the Hardy
potential on the existence and summability of solutions to a class of
nonlocal elliptic problems
$$
\left\{\begin{array}{rcll}
(-\Delta)^s u-\lambda \dfrac{u}{|x|^{2s}}&=&f(x,u) &\hbox{  in } \Omega,\\
u&=&0 &\hbox{  in } \mathbb{R}^N\setminus\Omega,\\ u&>&0 &\hbox{
in }\Omega,
\end{array}\right.
$$
where $(-\Delta)^s$, $s\in(0,1)$, is the fractional laplacian operator,
$\Omega\subset \ren$ is a bounded domain with Lipschitz boundary
such that $0\in\Omega$ and $N>2s$. We will mainly consider the solvability in
two cases:
\begin{enumerate}
\item The  linear problem, that is, $f(x,t)=f(x)$, where according to the
summability of the datum $f$ and the parameter $\lambda$ we give
the summability of the solution $u$.
\item The problem with a nonlinear term $f(x,t)=\frac{h(x)}{t^\sigma}$ for $t>0$. In this case, existence and regularity will depend on the value of $\sigma$ and on the summability of $h$.
\end{enumerate}

Looking for optimal results we will need a weak Harnack inequality for elliptic operators with \emph{singular coefficients} that seems to be new.
\end{abstract}

\maketitle

\section{Introduction and statement of the problem}
This work deals with the following problem
\begin{equation}\label{tlem}
\left\{\begin{array}{rcll}
(-\Delta)^s u-\lambda \dfrac{u}{|x|^{2s}}&=&f(x,u) &\hbox{  in } \Omega,\\
u&>&0 &\hbox{  in } \Omega,\\ u&=&0 &\hbox{ in }
\mathbb{R}^N\setminus\Omega,
\end{array}\right.
\end{equation}
where $s\in(0,1)$ is such that $2s<N$, $\Omega\subset \mathbb{R}^N$ is
a bounded regular domain containing the origin and $f$
is a measurable function satisfying suitable hypotheses.

Recall that we define the fractional Laplacian
$(-\Delta)^s $ as the operator given by the Fourier multiplier
$|\xi|^{2s}$. That is, for $u\in \mathcal{S}(\mathbb{R}^N)$,
$$\mathcal{F}((-\Delta)^s u
)(\xi):=|\xi|^{2s}\mathcal{F}(u)(\xi).$$

A computation involving the inverse Fourier transform
of a homogeneous tempered distribution gives the formal
expression  of the fractional Laplacian as an integral operator,
see for instance \cite{Stein}. More precisely, if
$u\in\mathcal{S}(\mathbb{R}^N)$,
\begin{equation}\label{operator}
(-\Delta)^{s}u(x):=a_{N,s}\mbox{ P.V. }\int_{\mathbb{R}^{N}}{\frac{u(x)-u(y)}{|x-y|^{N+2s}}\, dy},\, \,\qquad  s\in(0,1),
\end{equation}
where
\begin{equation} \label{const}
a_{N,s}:=\left(\int_{\mathbb{R}^{N}}{\dfrac{1-cos(\xi_1)}{|\xi|^{N+2s}}d\xi}\right)^{-1}=2^{2s-1}\pi^{-\frac
N2}\frac{\Gamma(\frac{N+2s}{2})}{|\Gamma(-s)|}.
\end{equation}

Due to its second term, problem \eqref{tlem} is related to the following Hardy
inequality, proved in \cite{He} (see also \cite{B, FLS, SW,Y}),
\begin{equation}\label{Hardy}
\dint_{\ren} \,|\xi|^{2s} \hat{u}^2\,d\xi\geq
\Lambda_{N,s}\,\dint_{\ren} |x|^{-2s} u^2\,dx\,\,\,\forall u\in
\mathcal{C}^{\infty}_{0}(\ren),
\end{equation}
where
\begin{equation}\label{cteHardy}
\Lambda_{N,s}:=
2^{2s}\dfrac{\Gamma^2(\frac{N+2s}{4})}{\Gamma^2(\frac{N-2s}{4})}
\end{equation}
 is optimal  and it  is not attained. Moreover,
 $$\lim_{s\to 1}\Lambda_{N,s}=\left(\dfrac{N-2}{2}\right)^2,$$
the classical Hardy constant.

The Hardy inequality \eqref{Hardy} plays an important role, for instance, in a general proof of
 the \textit{ stability of the relativistic matter}, see \cite{FLS}.
 We can also rewrite inequality \eqref{Hardy} in the
form
\begin{equation*}\label{hardy}
\frac{a_{N,s}}{2}\int_{\ren}\int_{\ren}{\frac{|u(x)-u(y)|^2}{|x-y|^{N+2s}}}\,
dx\, dy\geq
\Lambda_{N,s}\int_{\ren} {\frac{u^2}{|x|^{2s}}\,dx},\,u\in
\mathcal{C}^{\infty}_{0}(\ren),
\end{equation*}
which we will often use along the paper.
As we will see, this critical value $\Lambda_{N,s}$ will also play a fundamental role concerning solvability.
In particular, we already know thaht for $\lambda>\Lambda_{N,s}$ problem \eqref{tlem} has no positive
supersolution (see for example \cite{BMP, F}). Hence, from now on we will assume  $0<\lambda\leq \Lambda_{N,s}$.
\medskip

If $f(x,s)=f(x)$, problem \eqref{tlem} is
reduced to the linear case
\begin{equation}\label{prob}
\left\{\begin{array}{rcll}
(-\Delta)^s u-\lambda \dfrac{u}{|x|^{2s}}&=&f(x) &\hbox{  in } \Omega,\\
u&=&0 &\hbox{  in } \mathbb{R}^N\setminus\Omega,
\end{array}
\right.
\end{equation}
and our first goal will be to obtain the optimal summability of $u$ according
to the summability of the datum $f$ and the parameter $\lambda$. The case $\lambda=0$ can be found in \cite{LPPS}, where some Calderón-Zygmund type results are obtained.

Since this problem is linear, we will assume, without loss of generality,  that the datum $f$ is positive and we will deal with positive solutions.

The influence of the Hardy potential in the local case ($s=1$) was
studied in \cite{BOP}. The main results there can be summarized as follows. Suppose $f\in L^m(\Omega)$ and let
$$\lambda(m):=\dfrac{N(m-1)(N-2m)}{m^2}.$$
Then if $0<\lambda<\lambda(m)$ the solution to problem \eqref{prob} verifies the same Calderón-Zygmung  inequalities as for $\lambda=0$. On the contrary, if $\lambda\ge \lambda(m)$ it is possible to find counterexamples of these results, so the regularity does not hold (see \cite{BOP} for details).

In particular, if $m>\frac{N}2$ the solutions are unbounded, and if $m=1$ there is no solution in general. Indeed, the necessary and sufficient condition in order to have solvability is to assume the following integrability of the datum with respect to the weight,
$$
\dint_{\Omega} f |x|^{-\alpha_1}\,dx<\infty,
$$
where $\alpha_{1}:= \dfrac{N-2}{2}-\sqrt{\Big( \dfrac{N-2}{2}\Big)^2- \lambda}$ (see  \cite{AP1}).

In order to find an analogous optimal condition for problem \eqref{prob}, we will need a \textit{weak Harnack inequality} for a singular weighted nonlocal operator (that we will define in \eqref{Ltilde0}). This study, performed in Section 3, requires the combination of techniques on elliptic operators  and very involved computations on nonlocal radial integrals, and it provides the precise behavior of the solutions around the origin. This result will be the key in the proofs of existence and regularity in the next Sections.

%
%Also as a consequence, we will obtain the nonexistence of weak positive supersolution for $\lambda>\Lambda_{N,s}$ and positive datum $f$.
As an application and as a complement to the results in \cite{BMP}, we will also study a semilinear problem which is singular at the boundary. More precisely, we will consider  the problem
\begin{equation}\label{tlem1}
\left\{\begin{array}{rcll}
(-\Delta)^s u-\lambda \dfrac{u}{|x|^{2s}}&=&\dfrac{h(x)}{u^\sigma} &\hbox{  in } \Omega,\\
u&>&0 &\hbox{  in } \Omega,\\
u&=&0 &\hbox{  in } \mathbb{R}^N\setminus\Omega.\\
\end{array}
\right.
\end{equation}
The local case ($s=1$) with $\lambda=0$ was studied in \cite{BO}. Here the
authors proved that for all $h\in L^1(\Omega)$, there exists at
least one distributional solution. Regularity is
obtained according to the regularity of $h$ and the value of $\sigma$.
\medskip

The main purpose of this work is to obtain the same kind of results for the fractional Laplacian
framework, whose nonlocal behavior introduces new difficulties.
Some partial results have been already obtained in \cite{AA}, including  the $p$-Laplacian like operator.

The paper is organized as follows. In Section
\ref{prim} we precise the meaning  of solutions that will be used
along the work, with the corresponding functional setting. Some useful
tools as the Picone inequality, compactness results and certain
algebraic inequalities are also proved here.

In Section \ref{HH} we prove a
weighted singular version of the Harnack inequality. Notice that,
using the \emph{ground state transformation} stated in Lemma
\ref{GS}, the weak Harnack inequality gives the exact blow up rate
for the positive supersolutions to \eqref{tlem} near the
origin. As we said, this theorem will be the key for the optimality in the results of the following sections.

In Section \ref{s4} we treat the linear problem \eqref{prob}. According to $\lambda$ and the summability of $f$, we find the optimal summability of the solution $u$ for certain values of the spectral parameter $\lambda$. In particular, we see that the local techniques applied in \cite{BOP} do not give complete information in this framework, leaving the optimality for certain ranges of $\lambda$ as an open problem. We analyze this situation in detail in this section.

Finally, last section is devoted to study problem \eqref{tlem1}. We prove
existence and regularity results depending on the value of $\sigma$.

\section{Functional setting and useful tools}\label{prim}
Let $s\in (0,1)$. For any $p\in [1,\infty) $ and
$\Omega\subseteq \ren$, we define $W^{s,p}(\Omega)$ as follows,
$$
W^{s,p}(\Omega):= \left\{ u\in L^{p}(\Omega)\mbox{ s.t. } \,
\dfrac{|u(x)-u(y)|}{|x-y|^{\frac{N}{p}+s}} \in
L^{p}(\Omega\times\Omega)\right\}.
$$
We focus on the case $p=2$, where the fractional Sobolev spaces
$H^{s}(\Omega):=W^{s,2}(\Omega)$ turn out to be Hilbert spaces.
Moreover, if $\Omega=\ren$, the Fourier transform provides an alternative definition.
\begin{Definition}\label{Hs} {For  $0<s<1$,} we define  the fractional Sobolev space of order $s$ as
$$H^s(\mathbb{R}^N):=\{u\in L^2(\mathbb{R}^N)\,\mbox{ s.t. }\, |\xi|^{s}\mathcal{F}(u)(\xi)\in L^2(\mathbb{R}^N)\}.$$
\end{Definition}
Hence by Plancherel identity, we obtain a new expression for the norm of the Hilbert space $H^s(\mathbb{R}^N)$ (see \cite{FLS} for a detailed proof).
\begin{Proposition}\label{norma} Let $N\ge 1$ and $0<s<1$. Then for all $u\in H^s(\mathbb{R}^N)$
\begin{equation*}\label{integral}
\int_{\mathbb{R}^N} |\xi|^{2s}|\mathcal{F}(u)(\xi)|^2d\xi={\frac{a_{N,s}}{2}}\int_{\mathbb{R}^N}\int_{\mathbb{R}^N}\frac{|u(x)-u(y)|^2}{|x-y|^{N+2s}}dxdy,
\end{equation*}
where $a_{N,s}$ is the constant defined in \eqref{const}.
\end{Proposition}
Moreover we can extend by density the operator  $(-\Delta)^s u$ defined in \eqref{operator} from $\mathcal{S}(\mathbb{R}^N)$ to $H^s(\mathbb{R}^N)$. In this way, the associated scalar product  can be reformulated as follows
\begin{equation*}\begin{split}
\langle u,v\rangle_{ H^s(\mathbb{R}^N)}&:=\langle (-\Delta)^s u, v\rangle+(u,v)\\
&:=
P.V.\int_{\mathbb{R}^N}\int_{\mathbb{R}^N}{\frac{(u(x)-u(y))(v(x)-v(y))}{|x-y|^{N+2s}}\,dx\,dy}+\int_{\mathbb{R}^N}uv\,dx.
\end{split}\end{equation*}
We call $\|\cdot\|_{H_0^s (\mathbb{R}^N)}$ the induced norm by this scalar product.
Summarizing  the previous result we obtain the following useful formulation, that includes the corresponding \textit{integration  by parts} (see for instance \cite{dine}).
\begin{Proposition} Let $s\in (0,1)$ and $u\in H^{s}(\ren)$. Then,
$$
\frac{a_{N,s}}{2}\langle (-\Delta)^su,u\rangle= \frac{a_{N,s}}{2}
\|(-\Delta)^{\frac{s}{2}}u\|^{2}_{L^2(\ren)}=\||\xi|^s\mathcal{F}
u\|^{2}_{L^2(\ren)}.
$$
\end{Proposition}
The dual space of $H^{s}(\mathbb{R}^N)$ is defined by
$$H^{-s}(\mathbb{R}^N)=\{f\in \mathcal{S}'(\mathbb{R}^N)\,/\, |\xi|^{-s}\mathcal{F}(f)\in L^2(\mathbb{R}^N)\}.$$
The following properties are immediate:
\begin{enumerate}
\item $(-\Delta)^s :H^{s}(\mathbb{R}^N)\rightarrow H^{-s}(\mathbb{R}^N)$ is a continuous operator.
\item $(-\Delta)^s $ is a symmetric operator in $H^{s}(\mathbb{R}^N)$, that is,
$$\langle (-\Delta)^s u, v\rangle=\langle u, (-\Delta)^s v\rangle,\quad u,v\in H^{s}(\mathbb{R}^N).$$
\item  {Denoting also} by $\langle  \cdot  , \cdot  \rangle$ the natural duality product between
$H^{s}(\mathbb{R}^N)$ and $H^{-s}(\mathbb{R}^N)$, then
$$|\langle (-\Delta)^s u, v\rangle| \leq \| u
\|_{H^{s}(\mathbb{R}^N)} \|v\|_{H^{s}(\mathbb{R}^N) } .$$
\end{enumerate}
We define now the space $H_0^s (\Omega)$ as
the completion of $\mathcal{C}^\infty_0(\Omega)$ with respect to
the norm of $H^s(\R^N)$.
Notice that if  $u\in
H^s_0(\Omega)$, we have $u=0 \hbox{ a.e. in } \R^N\setminus
\Omega$ and we can write
$$
\int_{\ren}\int_{\ren}
\frac{|u(x)-u(y)|^2}{|x-y|^{N+2s}}\,dx\,dy=\iint_{D_\Omega}
{\frac{|u(x)-u(y)|^2}{|x-y|^{N+2s}}\,dx\,dy}
$$
where
$$
{D_\Omega} := \ren \times \ren \setminus \big( \mathcal{C} \Omega \times  \mathcal{C} \Omega \big) \,.
$$
\begin{remark} If $\Omega$ is a bounded domain and $u\in \mathcal{C}^\infty_0(\Omega)$, then by setting
$$|||u|||_{H_0^s (\Omega)}:=\left(\int_{\mathbb{R}^N} \int_{\mathbb{R}^N}{\frac{|u(x)-u(y)|^2}{|x-y|^{N+2s}}\,dx\,dy}\right)^{1/2}\,,
$$
and using a Poincaré type inequality  (see \cite{AMPP} or \cite{dine}), we can prove that
$|||\cdot|||_{H_0^s (\Omega)}$ and $\|\cdot\|_{H_0^s (\Omega)}$ are equivalent norms.
\end{remark}
If we denote by $H^{-s}(\Omega):=[H_0^s (\Omega)]^*$ the dual space
of $H_0^s (\Omega)$, then
$$(-\Delta)^s :H_0^s (\Omega)\rightarrow H^{-s}(\Omega),$$
is a continuous operator.

We give the meaning  of solutions that will be used along the paper: $i)$ {\it energy solutions} when the variational framework can be used
and $ii)$ {\it weak solutions} for data that are integrable but not in the dual space.

Recall that we assume $0\in \Omega$.
\begin{Definition}\label{energia}
Assume $0<\lambda < \Lambda_{N,s}$. For $f \in  H^{-s}
(\Omega) $ we say that $u\in  H_{0}^{s}(\Omega)$ is a {\it finite
energy solution} to \eqref{tlem} if
$$
\frac{a_{N,s}}{2}\langle (-\Delta)^s u, w\rangle-\lambda\,\dint_{\Omega} \dfrac{u w}{|x|^{2s}}\,dx
=\langle f,w\rangle ,  \qquad
 \forall w\in H_0^s (\Omega)\,.$$
\end{Definition}
If $\lambda<\Lambda_{N,s}$, then existence and uniqueness of a
solution $u\in H_{0}^{s}(\Omega)$ for all $f\in
H^{-s}(\Omega)$ easily follows.
\begin{remark}
If $\lambda=\Lambda_{N,s}$, the same result holds but in a
space $H(\Omega)$ defined as the completion of
$\mathcal{C}^\infty_0(\Omega)$ with respect to the norm
$$
\|\phi\|^2_{H(\Omega)}:=\frac{a_{N,s}}{2}\iint_{D_\Omega}{\frac{|\phi(x)-\phi(y)|^2}{|x-y|^{N+2s}}}\,
dx\, dy-\Lambda_{N,s}\int_{\Omega}{\frac{\phi^2}{|x|^{2s}}\,dx}.
$$
By using the improved Hardy inequality (see for instance \cite{APP}) we get that $H(\Omega)$ is a Hilbert space and $H^s_0(\Omega)\varsubsetneq H(\Omega)\subsetneqq
W^{s,q}_0(\Omega)$, for all $q<2$.
\end{remark}

To deal with the case of a general  $f\in L^1(\Omega)$, we need to define the notion of {\it weak solution}, where we only request the regularity needed to give weak sense to the
equation.

Since the operator is nonlocal, we need to precise the class of test function to be considered, that precisely is,
\begin{equation}\label{test}\begin{split}
\mathcal{T}:=\{\phi:\mathbb{R}^N\rightarrow\mathbb{R}\,|\, (-\Delta)^s \phi=\varphi,\, \varphi\in L^\infty(\Omega)\cap \mathcal{C}^\alpha(\Omega), 0<\alpha<1, \phi=0\inn \ren\setminus \Omega\}.
\end{split}\end{equation}
Notice that every $\phi\in \mathcal{T}$ belongs in particular to $L^\infty(\Omega)$ (see \cite{LPPS}) and moreover it is a strong solution to the equation $(-\Delta)^s \phi=\varphi$. See for instance \cite{SEV} and \cite{Silvestre}.

\begin{Definition}\label{veryweak}
Assume  $f\in L^{1}(\Omega)$. We  say that $u\in {L}^{1}(\O)$ is a weak supersolution (subsolution) of problem \eqref{tlem} if $\dfrac{u}{|x|^{2s}}\in L^1(\Omega)$, $u=0\inn \ren\setminus \Omega$, and for all nonnegative $\phi\in
\mathcal{T}$,  the following inequality holds,
\begin{equation*}\label{eq:subsuper}
\dint_{\Omega} u (-\Delta)^s \phi\,dx -\lambda\dint_{\Omega} \dfrac{u\phi}{|x|^{2s}}\,dx\ge (\le)\\
\dint_{\Omega}\,f\phi\,dx.
\end{equation*}
If $u$ is super and subsolution, then we say that $u$
is a weak solution.
\end{Definition}
Notice that, if $u\in \mathcal{C}_{0}^{\infty} (\mathbb{R}^{N})$,  Frank, Lieb and Seiringer proved in
\cite{FLS} the following result.
\begin{Lemma}\label{GS}(Ground State Representation) Let $0<\gamma<
\frac{N-2s}{2}$. If $u\in \mathcal{C}_{0}^{\infty} (\mathbb{R}^{N})$ and
$v(x):= |x|^{\gamma} u(x)$, then
$$
\dint_{\mathbb{R}^{N}} \, |\xi|^{2s} |\hat{u}(\xi)|^{2}\,d\xi
-(\Lambda_{N,s}+\Phi_{N,s}(\gamma))
\dint_{\mathbb{R}^{N}}|x|^{-2s}|u(x)|^2\,dx=a_{N,s}\iint_{\mathbb{R}^{2N}}\,
\frac{|v(x)-v(y)|^2}{|x-y|^{N+2s}}\frac{dx}{|x|^{\gamma}}
\frac{dy}{|y|^{\gamma}},
$$
where $$\Phi_{N,s}(\gamma)=2^{2
s}\left(\frac{\Gamma\big(\frac{\gamma+2s}{2}\big)\Gamma\big(\frac{N-\gamma}{2}\big)}{\Gamma\big(\frac{N-\gamma-2s}{2}\big)\Gamma\big(\frac{\gamma}{2}\big)}-
\frac{\Gamma^2\big(\frac{N+2s}{4}\big)}{\Gamma^2\big(\frac{N-2s}{4}\big)}\right).$$
\end{Lemma}
Notice that in particular this representation  proves that the constant $\Lambda_{N,s}$ is optimal and is not attained. See \cite[Remark 4.2]{FLS}  for details.

Using this representation with $\lambda=\Lambda_{N,s}+\Phi_{N,s}(\gamma)$, we obtain that
if $u$ solves
$$
\left\{\begin{array}{rcll}
(-\Delta)^s u-\lambda \dfrac{u}{|x|^{2s}}&=&f(x,u) &\hbox{  in } \Omega,\\
u&=&0 &\hbox{  in } \mathbb{R}^N\setminus\Omega,
\end{array}\right.
$$
then $v$ satisfies
\begin{equation}\label{equi}
\left\{\begin{array}{rcll}
L_\gamma v&=&|x|^{-\gamma}f(x, |x|^{-\gamma} v)=: g(x,v)&\inn\Omega,\\
v&=&0&\inn \mathbb{R}^{N}\setminus\Omega,
\end{array}\right.
\end{equation}
with
\begin{equation}\label{Ltilde0}
L_\gamma v:=a_{N,s} \mbox{ P.V. }\dint_{\mathbb{R}^{N}}
\dfrac{v(x)-v(y)}{|x-y|^{N+2s}}\,\dfrac{dy}{|x|^{\gamma}|y|^{\gamma}}.
\end{equation}
Observe that if $\gamma\rightarrow 0$, then $\Phi_{N,s}(\gamma)\rightarrow -\Lambda_{N,s}$ and $\lambda\rightarrow 0$. On the other hand, if $\gamma\rightarrow \dfrac{N-2s}{2}$, then $\Phi_{N,s}(\gamma)\rightarrow 0$ and $\lambda\rightarrow \Lambda_{N,s}$.

To analyze the behavior and the regularity of $u$, we  deal with the same questions for $v$. Thus, we need
to work in fractional Sobolev spaces with admissible weights. For
simplicity of typing, we denote
$$d\mu:=
\dfrac{dx}{|x|^{2\gamma}} \quad \hbox{ and } \quad d\nu:=
\dfrac{dxdy}{|x-y|^{N+2s}|x|^\gamma|y|^\gamma}.$$

\noindent   For $\O\subseteq \ren$, we define the weighted fractional Sobolev space
$Y^{s,\gamma}(\Omega)$ as follows
$$
Y^{s,\gamma}(\Omega)\dyle := \Big\{ \phi\in
L^2(\Omega,d\mu)\mbox{ s.t. }\dint_{\Omega}\dint_{\Omega}(\phi(x)-\phi(y))^2d\nu<+\infty\Big\}.
$$
It is clear that $Y^{s,\gamma}(\O)$ is a Hilbert space endowed with
the norm
$$
\|\phi\|_{Y^{s,\gamma}(\Omega)}:=
\Big(\dint_{\Omega}|\phi(x)|^2d\mu
+\dint_{\Omega}\dint_{\Omega}(\phi(x)-\phi(y))^2d\nu\Big)^{\frac
12}.
$$

The following  extension lemma can be proved by using the
same arguments of \cite{Adams} (see also \cite{dine}).
\begin{Lemma}\label{ext}
Let $\Omega\subset \ren$ be a smooth domain. Then for all $w\in
Y^{s,\g}(\Omega)$, there exists $\tilde{w}\in Y^{s,\g}(\ren)$ such
that $\tilde{w}_{|\Omega}=w$ and
$$
\|\tilde{w}\|_{Y^{s,\g}(\ren)}\le C \|{w}\|_{Y^{s,\g}(\Omega)},
$$
where $C:= C(N,s,\O, \gamma)>0$.
\end{Lemma}
We define the space $Y^{s,\gamma}_0(\Omega)$ as the
completion of $\mathcal{C}^\infty_0(\Omega)$ with respect to the
norm of $Y^{s,\g}(\Omega)$. If  $\Omega$ is a bounded regular domain, then as in the case
$\gamma=0$, we have the next Poincar\'e inequality (a proof
can be found in the Appendix B of \cite{AMPP}).
\begin{Theorem}\label{poincare}
There exists a positive constant $C:=C(\O,N,s  {,\gamma})$ such
that for all $\phi\in \mathcal{C}^\infty_0(\O)$, we have
$$ C\dint_{\O}\phi^2(x)d\mu\le \int_{\O}
\int_{\O}(\phi(x)-\phi(y))^2d\nu.$$
\end{Theorem}

As a consequence we reach that if $\Omega$ is a bounded domain, we
can consider $Y^{s,\gamma}_0(\Omega)$ with the equivalent norm
$$
|||\phi|||_{Y^{s,\gamma}_0(\Omega)}:=
\Big(\dint_{\Omega}\dint_{\Omega}(\phi(x)-\phi(y)|^2d\nu\Big)^{\frac
12}.
$$
In \cite{AB}, the authors prove the following weighted
Sobolev inequality.
\begin{Proposition} \label{Sobolev}
Consider $0<s<1$ such that $N>2s$ and $0<\gamma<\dfrac{N-2s}{2}$.
Then, for all $v\in \mathcal{C}_{0}^{\infty}(\ren)$, there exists
a positive constant $S=S(N,s,\g)$ such that
$$
\frac{a_{N,s}}{2}\dint_{\mathbb{R}^{N}}\dint_{\mathbb{R}^{N}}
\dfrac{|v(x)-v(y)|^{2}}{|x-y|^{N+2s}}\frac{dx}{|x|^{\gamma}}\,
\frac{dy}{|y|^{\gamma}}\geq S \left(\dint_{\mathbb{R}^{N}}
\dfrac{|v(x)|^{2_{s}^{*}}}{|x|^{2_{s}^{*}\gamma}}\right)^{\frac{2}{2^{*}_{s}}},
$$
where $2^{*}_{s}:= \dfrac{2N}{N-2s}$.

If $\Omega \subset \ren$ is a bounded domain and
$\gamma=\dfrac{N-2s}{2}$, then for all $q<2$, there exists a
positive constant $C=C(\Omega, \gamma, s,q)$ such that
$$
\frac{a_{N,s}}{2}\dint_{\mathbb{R}^{N}}\dint_{\mathbb{R}^{N}}
\dfrac{|v(x)-v(y)|^{2}}{|x-y|^{N+2s}}\frac{dx}{|x|^{\gamma}}\,
\frac{dy}{|y|^{\gamma}}\geq C\left(\dint_{\mathbb{R}^{N}}
\dfrac{|v(x)|^{2_{s,q}^{*}}}{|x|^{2_{s,q}^{*}\gamma}}\right)^{\frac{2}{2_{s,q}^{*}}},
$$
for all $v\in \mathcal{C}_{0}^{\infty}(\Omega)$, where
$2^{*}_{s,q}:= \dfrac{2N}{N-qs}$.
\end{Proposition}
Combining the previous proposition and the extension lemma we get the
next Sobolev inequality in the space $Y^{s,\gamma}_0(\Omega)$.

\begin{Proposition} \label{Sobolev00}
Let $\Omega$ be a bounded regular domain and suppose that the
hypotheses of Proposition \ref{Sobolev} hold, then, for all $v\in
\mathcal{C}_{0}^{\infty}(\Omega)$, there exists a positive
constant $S=S(N,s,\g, \Omega)$ such that
$$
\frac{a_{N,s}}{2}\dint_{\Omega}\dint_{\Omega}
\dfrac{|v(x)-v(y)|^{2}}{|x-y|^{N+2s}}\frac{dx}{|x|^{\gamma}}\,
\frac{dy}{|y|^{\gamma}}\geq S \left(\dint_{\Omega}
\dfrac{|v(x)|^{2_{s}^{*}}}{|x|^{2_{s}^{*}\gamma}}\right)^{\frac{2}{2^{*}_{s}}}.
$$
\end{Proposition}
\begin{proof}
Let $v\in \mathcal{C}^\infty_0(\Omega)$ and define $\tilde{v}$ to
be the extension of $v$ to $\ren$ given in Lemma \ref{ext}. Then
using Proposition \ref{Sobolev}, we get
$$
C \|{v}\|_{Y^{s,\g}(\Omega)}\ge
\frac{a_{N,s}}{2}\dint_{\ren}\dint_{\ren}\dfrac{|\tilde{v}(x)-\tilde{v}(y)|^{2}}{|x-y|^{N+2s}}\dfrac{dxdy}{|x|^\g|y|^\g}\ge
S(N,s,\gamma)\left(\dint_{\ren}\frac{|\tilde{v}(x)|^{2^*_s}}{|x|^{2^*_s\gamma}}dx\right)^{\frac{2}{2^{*}_{s}}}.$$
Since $\tilde{v}_{|\Omega}=v$, then using Theorem \ref{poincare},
we reach that
$$
\frac{a_{N,s}}{2}\dint_{\Omega}\dint_{\Omega}\dfrac{|{v}(x)-{v}(y)|^{2}}{|x-y|^{N+2s}}\dfrac{dxdy}{|x|^\g|y|^\g}\ge
S(N,s,\gamma,
\Omega)\left(\dint_{\Omega}\frac{|{v}(x)|^{2^*_s}}{|x|^{2^*_s\gamma}}dx\right)^{\frac{2}{2^{*}_{s}}}
$$ and the result follows.
\end{proof}
We state now a weighted version of the Poincaré-Wirtinger
inequality that we will use later (see Appendix B in \cite{AMPP} for a proof).
\begin{Theorem}\label{PW}
Let $r>0$ and $w\in
Y^{s,\g}(B_r)$ and assume that $\psi$ is a radial
decreasing function such that $\text{supp}\:\psi\subset B_r$ and $0\lneqq \psi\le 1$. Define
$$
W_\psi:=\dfrac{\int_{B_r}w(x)\psi(x)d\mu}{\int_{B_r}\psi(x)d\mu}.
$$
Then,
 $$ \int_{B_r}(w(x)-W_\psi)^2\psi(x)d\mu\le C r^{2s}\int_{B_r}
\int_{B_r}(w(x)-w(y))^2\min\{\psi(x),
\psi(y)\}d\nu.
$$
\end{Theorem}
\noindent Finally, we define
$$Y^{s,\gamma}_{loc}(\Omega):=\{ u \in L_{loc}^2(\Omega,d\mu)\,\mbox{ s.t. }\, \forall\, \Omega_1\subset\subset \Omega,
\int_{\Omega_1}\int_{\Omega_1}{\frac{|u(x)-u(y)|^2}{|x-y|^{N+2s}}}\,
\frac{dx}{|x|^\g}\frac{dy}{|y|^\g}+\int_{\Omega_1}u^2 d\mu<+\infty\}.$$
\begin{remark}
This definition  for $\gamma=0$ is similar. In such a case we will denote the associated space as $H_{loc}^s(\Omega)$.
\end{remark}
We consider the following natural definition.
\begin{Definition}\label{super00}
Let $v\in Y^{s,\g}_{loc}(\Omega)$. We say that $v$ is a
supersolution to problem \eqref{equi} if
\begin{equation}\label{super}
\iint_{{D_{\Omega_1}}}(v(x) - v(y))(\varphi(x)- \varphi(y))d\nu
\ge \dint_{\Omega_1}\,g\varphi\,dx\end{equation} for every
nonnegative $ \varphi \in Y_0^{s,\g} (\Omega_1)$ and every
$\O_1\subset\subset \O$.
\end{Definition}
An integral extension involving positive Radon measures of a well-known punctual  inequality by Picone (see \cite{Pi}) was obtained in \cite{AP} in the local framework. An extension to the fractional    setting has been obtained in \cite{LPPS}. A similar inequality holds for the operator
\begin{equation}\label{LgO}
L_{\g, \O}(w)(x):=a_{N,s}\mbox{P.V.}\int_\O\frac{w(x)-w(y)}{|x-y|^{N+2s}}\frac{dy}{|x|^\g|y|^\g}.
\end{equation}
Notice that, if $\Omega=\R^N$, $L_{\g,\R^N}$ coincides with $L_\g$ defined in \eqref{Ltilde0}.
\begin{Theorem}{\it (Picone's type Inequality).}\label{Picone}
Let $w\in Y^{s,\g}(\O)$ be such that $w>0$ in $\O$, and assume that
$L_{\g, \O}(w)= \nu$ with $\nu\in L^1_{loc}(\ren)$ and $\nu\gneqq
0$. Then for all $v\in \mathcal{C}^\infty_0(\O)$ we have
\begin{equation}\label{picone1}
\frac {a_{N,s}}{2}
\dint_{\O}\dint_{\O}\dfrac{|v(x)-v(y)|^{2}}{|x-y|^{N+2s}}\dfrac{dxdy}{|x|^\g|y|^\g}\ge
\langle L_{\g,\O} (w),\frac{v^2}{w}\rangle
\end{equation}
where  $ L_{\gamma, \Omega}w$ is defined by \eqref{LgO}.
\end{Theorem}
The proof is the same as in \cite{LPPS}, where is based in a punctual inequality.
As a consequence, we have the next comparison principle that
extends to the weighted fractional framework the classical one obtained
by Brezis and Kamin in \cite{BK}.

\begin{Lemma}\label{compar}
Let $\O$ be  a bounded domain and let $f$ be a nonnegative
continuous function such that $f(x, \sigma)>0$ if $\sigma>0$ and
$\dfrac{f(x,\sigma)}{\sigma}$ is decreasing. Let $u,v\in
Y^{s,\gamma}_0(\O)$ be such that $u,v>0$ in $\O$ and
$$\left\{
\begin{array}{rcl}
L_{\g,\O}(u) &\geq & f(x, u)\mbox{  in  }\O,\\
L_{\g,\O} (v) & \le & f(x,v)\mbox{  in   }\O.
\end{array}
\right.
$$
Then, $u\geq v\inn\Omega$.
\end{Lemma}
\begin{proof}
The proof is the same as in the case of constant coefficients, since it relies on several pointwise inequalities (see \cite{LPPS} for details).
\end{proof}
In the sequel we will need the next compactness result.
\begin{Lemma}\label{compa}
Let $\{u_n\}_{n\in\N}\subset Y^{s,\g}_0(\O)$ be an increasing  sequence of nonnegative
functions such that $L_{\g,\O} (u_n)\ge 0$. Assume that
$u_n\rightharpoonup u$ weakly in $Y^{s,\g}_0(\O)$. Then, $u_n\to u$
strongly in $Y^{s,\g}_0(\O)$.
\end{Lemma}
\begin{proof}
Since $L_{\g,\O}(u_n)\ge 0$, then $\langle L_{\g,\O} (u_n),
w_n\rangle \le 0$, where $w_n:= u_n-u$. Thus
$$
\dint_{\O}\dint_{\O}\dfrac{(u_n(x)-u_n(y))(w_n(x)-w_n(y))}{|x-y|^{N+2s}}\dfrac{dxdy}{|x|^\g|y|^\g}\le
0.
$$
Since
$$
(u_n(x)-u_n(y))(w_n(x)-w_n(y))=(u_n(x)-u_n(y))^2-(u_n(x)-u_n(y))(u(x)-u(y))
$$
 by Young inequality we conclude that
$$
\dint_{\O}\dint_{\O}\dfrac{(u_n(x)-u_n(y))^2}{|x-y|^{N+2s}}\dfrac{dxdy}{|x|^\g|y|^\g}\le
\dint_{\O}\dint_{\O}\dfrac{(u(x)-u(y))^2}{|x-y|^{N+2s}}\dfrac{dxdy}{|x|^\g|y|^\g}.
$$
Therefore  $\|u_n\|_{Y^{s,\g}_0(\O)}\le
\|u\|_{Y^{s,\g}_0(\O)}$ and hence $u_n\to u$ strongly in
$Y^{s,\g}_0(\O)$.
\end{proof}
Likewise, we have  the following local version of Lemma \ref{compa}.
\begin{Lemma}\label{compag}
Let $\{u_n\}_{n\in\N}\subset Y^{s,\g}_0(\O)$ be an increasing sequence of nonnegative
 functions such that $L_{\g,\O} (u_n)\ge 0$. Assume that
$\{u_n\}_{n\in\N}$ is uniformly bounded in $Y^{s,\g}_{loc}(\O)$, then
there exists $u\in Y^{s,\g}_{loc}(\O)$ such that $u_n\to u$
strongly in $Y^{s,\g}_{loc}(\Omega)$.
\end{Lemma}
\begin{proof}
By using a straightforward modification of Lemma 5.3 in
\cite{dine}, and multiplying the sequence
$\{u_n\}_{n\in\mathbb{N}}$ by a Lipschitz cut-off function $\psi$
such that $\psi\equiv 1$ on $\Omega'\subset\Omega$, we can apply
Lemma \ref{compa} to conclude.
\end{proof}
\begin{remark}
Lemma \ref{compar}, Lemma \ref{compa} and Lemma \ref{compag} also hold for $\gamma=0$, i.e., for the spaces $H_0^s(\Omega)$ and $H^s_{loc}(\Omega)$.
\end{remark}
We will also  consider nonvariational data, i.e., in $L^1(\Omega)$. In this case, we will use the classical truncating procedure to get
\textit{a priori} estimates. Recall that for any $k\geq 0$, $T_k (\sigma) $ and $G_k (\sigma)$,
$\sigma\in \mathbb{R}^+$ are defined by
\begin{equation}\label{Tk}
T_k (\sigma) := {\min\{ k, \sigma \}} \quad \mbox{ and } \quad
  G_k (\sigma) := \sigma- T_k (\sigma)\,.
\end{equation}
\begin{Proposition}\label{prop}
Assume that $v\in H_0^s  (\Omega)$:
\begin{itemize}
\item[i)] if $\psi \in \lip (\mathbb{R})  $ is such that
$\psi(0)=0$, then $\psi (v) \in H_0^s  (\Omega)$. In particular,
for any $k\geq 0$, $T_k  (v)$, $ G_k (v)$ $\in H_0^s  (\Omega)$\,;
\item[ii)] for any $k \geq 0 $,
\begin{equation} \label{g_k}
   \,\| G_k  (v) \|^2_{H_0^s  (\Omega)}  \leq \int_\Omega G_k (v)\  (-\Delta)^s v\,dx\,;
\end{equation}
\item[iii)] for any $k \geq 0 $,
\begin{equation} \label{t_k}
    \, \| T_k  (v) \|^2_{H_0^s  (\Omega )} \leq  \int_\Omega T_k (v) \ (-\Delta)^s v\,dx.
\end{equation}
\end{itemize}
\end{Proposition}
A detailed proof of this result can be seen in \cite{LPPS}. Since the proof relies in a punctual inequality a similar result holds for the \textit{weighted operator.}

The next elementary algebraic inequality  will be used in some arguments. See \cite{Kas1} and \cite{AAB}.

For the reader convenience we give  a complete proof here.
\begin{Lemma}\label{alg}
Let $s_1,s_2\ge 0$ and $a>0$. Then
\begin{equation}\label{alg0}
(s_1-s_2)(s^a_1-s^a_2)\ge
\frac{4a}{(a+1)^2}(s^{\frac{a+1}{2}}_1-s^{\frac{a+1}{2}}_2)^2.
\end{equation}
\end{Lemma}
\begin{proof} Since $\dfrac{4a}{(a+1)^2}\le 1$ for $0\le a$, if  $s_1=0$ or $s_2=0$ the inequality trivially follows. Hence, we can assume $s_1>s_2>0$. Thus, setting
$x:=\dfrac{s_2}{s_1}$, \eqref{alg0} is equivalent to
\begin{equation}\label{alg00}
(1-x)(1-x^a)\ge \frac{4a}{(a+1)^2}(1-x^{\frac{a+1}{2}})^2\mbox{
for all  }x\in (0,1).
\end{equation}
We set
$$
h(x):=(1-x)(1-x^a)(a+1)^2-4a(1-x^{\frac{a+1}{2}})^2,
$$
and then we just have to show that $h(x)\ge 0$ for all $x\in (0,1)$. Moreover, $h$ can be written as
$$
h(x)=(a-1)^2(1-x^{\frac{a+1}{2}})^2-(a+1)^2(x^{\frac
12}-x^{\frac{a}{2}})^2.
$$

First, we assume $a>1$. We claim that
$$
(a-1)(1-x^{\frac{a+1}{2}})\ge (a+1)(x^{\frac 12}-x^{\frac{a}{2}}).
$$
In fact, let us define
$$
h_1(x):=(a-1)(1-x^{\frac{a+1}{2}})-(a+1)(x^{\frac
12}-x^{\frac{a}{2}}),
$$
so that
$h'_1(x)=\frac{(a+1)}{2}\Big(-(a-1)x^{\frac{a-1}{2}}-x^{-\frac{1}{2}}+ax^{\frac{a-2}{2}}\Big)$.

Using Young inequality, we obtain that
$$
x^{\frac{a}{2}-1}\le \frac{a-1}{a}x^{\frac{a-1}{2}}+
\frac{1}{a}x^{-\frac{1}{2}}.
$$
Thus $h'_1(x)\le 0$ and hence $h_1(x)\ge h_1(1)=0$. Therefore
$h(x)\ge 0$ and the result follows in this case.

Consider now the case $a<1$. We prove the result if we show
$$
(1-a)(1-x^{\frac{a+1}{2}})\ge (a+1)(x^{\frac{a}{2}}-x^{\frac 12}).
$$
Defining
$$
h_2(x):=(1-a)(1-x^{\frac{a+1}{2}})-(a+1)(x^{\frac{a}{2}}-x^{\frac
12})
$$
and using again Young inequality we obtain that $h'_2(x)\le 0$
for all $x\in (0,1)$. Thus $h_2(x)\ge h_1(1)=0$. Hence $h(x)\ge 0$
and we conclude.
\end{proof}
Finally we state  the following classical numerical iteration result proved in \cite{St}) and  that we will use later for some
boundedness results.
\begin{Lemma}\label{st}\sl
Let $\psi: \R^{+} \to \R^{+}$ be a nonincreasing function such
that
$$
\psi(h) \leq \frac{M\,\psi(k)^{\delta}}{(h-k)^{\gamma}}, \quad
\forall h > k > 0,
$$
where $M > 0$, $\delta > 1$ and $\gamma > 0$. Then $\psi(d) = 0$,
where $ d^{\gamma} = M\,\psi(0)^{\delta -
1}\,2^{\frac{\delta\gamma}{\delta -1}}. $
\end{Lemma}

\section{Weak Harnack inequality and local behavior of nonnegative supersolutions.}\label{HH}
Consider the homogeneous equation
\begin{equation*}\label{problemHardy}
(-\Delta)^{s} u-\l\dfrac{\,u}{|x|^{2s}}=0\mbox{ in } \mathbb{R}^N.
\end{equation*}
First, the following result holds (see \cite{BMP, F}).
\begin{Lemma}\label{singularity} Let $0<\lambda\leq \Lambda_{N,s}$. Then $u_{\pm\alpha}:=|x|^{-\frac{N-2s}{2}\pm\alpha}$
solves
\begin{equation*}\label{homogeneous}
(-\Delta)^s u= \lambda\frac{u}{|x|^{2s}}\inn
\ren\setminus{\{0\}},
\end{equation*}
where $\alpha$ is given  by the identity
\begin{equation}\label{lambda}
\lambda=\lambda(\alpha)=\lambda(-\alpha)=\dfrac{2^{2s}\,\Gamma(\frac{N+2s+2\alpha}{4})\Gamma(\frac{N+2s-2\alpha}{4})}{\Gamma(\frac{N-2s+2\alpha}{4})\Gamma(\frac{N-2s-2\alpha}{4})}.
\end{equation}
Moreover $\lambda(\alpha)$ is a positive decreasing continuous
function from $\left[0,\dfrac{N-2s}{2}\right)$ to $(0, \Lambda_{N_s}]$.
\end{Lemma}
\begin{remark}
Notice that $\lambda(\alpha)=
\lambda(-\alpha)=m_{\alpha}m_{-\alpha}$, with $m_{\alpha}:=
2^{\alpha+s}\dfrac{\Gamma(\frac{N+2s+2\alpha}{4})}{\Gamma(\frac{N-2s-2\alpha}{4})}$.
\end{remark}
\noindent Denote
\begin{equation}\label{g1}
\gamma:= \dfrac{N-2s}{2}-\alpha \hbox{ and } \bar\gamma:=
\dfrac{N-2s}{2}+\alpha,
\end{equation}
with $0<\gamma\leq \dfrac{N-2s}{2}\leq\bar\gamma<(N-2s)$. Since
$N-2\gamma-2s={2}\alpha>0$ and $N-2\bar\gamma-2s=-{2}\alpha<0$,
then $|x|^{-\gamma}$ is the unique energy solution of these ones
such that  $(-\Delta)^{s/2}(|x|^{-\g})\in L^2_{loc}(\R^N)$.

 Let $u$ be the energy solution to problem \eqref{prob} with $0<\lambda< \Lambda_{N,s}$. By setting $v(x):=|x|^{\g} u(x)$, where $\g$ is defined in
\eqref{g1}, it follows that $v$ solves
\begin{equation}\label{peso1}
\left\{
\begin{array}{rclc}
L_\g v(x) &= &|x|^{-\g}f(x) &\mbox{ in } \Omega,\\
v&=&0&\inn \ren\setminus\Omega,
\end{array}
\right.
\end{equation}
with  $0<\g<\frac{N-2s}{2}$ and $L_\gamma$ defined in \eqref{Ltilde0}. Hence, to study the behavior of
$u$ near the origin, we may deal with the same question for
$v$. More precisely,  we want to prove that the weighted operator $L_\g v$
satisfies a suitable weak Harnack inequality. Notice that the
natural functional framework for the new equation of $v$ is the
space $Y^{s,\g}(\R^N)$ defined in Section \ref{prim}.

The statement of the result is the
following.
\begin{Theorem}{\it(Weak Harnack inequality)}\label{harnack}\\
Let $r>0$ such that $B_{2r}\subset\Omega$. Assume that $f\ge 0$ and let $v\in Y^{s,\gamma}(\ren)$ be  a
supersolution to \eqref{peso1} with $v\gneqq 0$ in $\ren $. Then,
for every $q<\frac{N}{N-2s}$ there exists a positive constant $C=C(N,s,\g)$ such that
\begin{equation*}\label{main}
\Big(\int_{B_r}v^qd\mu\Big)^{\frac 1q}\le C
\inf_{B_{\frac{3}{2}r}}v.
\end{equation*}
\end{Theorem}

The proof follows  classical arguments by Moser  and
Krylov-Safonov (see \cite{FKS} for the local case with weights). For the nonlocal
case we have the precedent of \cite{CKP}, where the
kernel is comparable to a fractional Laplacian and the operator
considered is of fractional $p$-Laplacian type. Since the kernel
defined in \eqref{Ltilde0} is singular we have to check the arguments step by
step. That is, our result can be seen as the
fractional counterpart of \cite{FKS}.
Notice that it is enough to consider the case
$B_r(x_0)=B_r(0)$. For simplicity of  typing, we will write $B_r$
instead of $B_r(0)$. We start proving the following estimate.
\begin{Lemma}\label{lema1}
Let $R>0$ such that $B_R\subset\Omega$, and assume that $v\in Y^{s,\g}(\ren)$ with $v\gvertneqq 0$, is a
supersolution to \eqref{peso1}. Let $k>0$ and suppose that for
some $\sigma\in (0,1]$ we have
\begin{equation}\label{elli1}
|B_r\cap \{v\ge k\}|_{d\mu}\ge \sigma|B_r|_{d\mu}
\end{equation}
with $0<r<\frac{R}{16}$. Then there exists a positive constant
$C=C(N,s,\g)$ such that
\begin{equation*}\label{elli2}
|B_{6r}\cap \{v\le 2\delta k\}|_{d\mu}\le
\frac{C}{\sigma\log(\frac{1}{2\delta})}|B_{6r}|_{d\mu}
\end{equation*}
for all $\delta\in (0,\frac 14)$.
\end{Lemma}
\begin{proof} Without loss of generality we can assume that $v>0$ in
$B_R$, (otherwise we can deal with $v+\e$ and let $\e\to 0$ at the end). Let
$\psi\in \mathcal{C}^\infty_0(B_R)$ be such that $0\le \psi\le 1$,
$\text{supp}\: \psi\subset B_{7r}$, $\psi=1$ in $B_{6r}$ and
$|\nabla \psi|\le \frac{C}{r}$.

Using $\psi^{2}v^{-1}$ as a test function in \eqref{peso1}, it
follows that
$$
\int_{\mathbb{R}^{N}}\int_{\mathbb{R}^{N}}
(v(x)-v(y))(\psi^{2}(x)v^{-1}(x)-\psi^{2}(y)v^{-1}(y)) d\nu\ge 0.
$$
Thus
\begin{eqnarray}\label{eq1}
& 0\dyle \le \int_{B_{8r}}\int_{B_{8r}}
(v(x)-v(y))\left(\frac{\psi^{2}(x)}{v(x)}-\frac{\psi^{2}(y)}{v(y)}\right)
d\nu + 2\int_{\ren \setminus
B_{8r}}\int_{B_{8r}}(v(x)-v(y))\frac{\psi^{2}(x)}{v(x)}d\nu.
\end{eqnarray}
Denote $x=|x|x'$ and $y=\rho y'$, where $|x'|=|y'|=1$ and $\rho:=|y|$. We have that
\begin{equation*}\begin{split}
\dyle \int_{\ren \setminus
B_{8r}}\int_{B_{8r}}(v(x)-v(y))&\frac{\psi^{2}(x)}{v(x)}d\nu\le
\dyle\int_{\ren \setminus B_{8r}}\int_{B_{8r}}\psi^{2}(x)d\nu\\
&\le \dyle  \dint_{B_{7r}}
\frac{\psi^2(x)}{|x|^{\g}}\int_{8r}^{\infty}
\dfrac{\rho^{N-\g-1}}{|x|^{N+2s}}
\Big(\dint_{{\mathbb{S}}^{N-1}}\dfrac{dy'}{|\dfrac{\rho}{|x|}y'-x'|^{N+2s}}\Big)d\rho\,dx.
\end{split}\end{equation*}
Setting here $\tau:=\dfrac{\rho}{|x|}$,
\begin{equation*}\begin{split}
\dyle \int_{\ren \setminus
B_{8r}}\int_{B_{8r}}(v(x)-v(y))&\frac{\psi^{2}(x)}{v(x)}d\nu \le
C\dint_{B_{7r}}\frac{\psi^{2}(x)}{|x|^{2\g+2s}}\int_{\frac
87}^{\infty} \tau^{N-\g-1}
\Big(\dint_{{\mathbb{S}}^{N-1}}\dfrac{dy'}{|\tau y'-x'|^{N+2s}}\Big)d\tau\,dx\\
& \le  C\dyle
\dint_{B_{7r}}\frac{\psi^{2}(x)}{|x|^{2\g+2s}}\int_{\frac
87}^{\infty} \tau^{N-\g-1}D(\tau)d\tau\,dx
\end{split}\end{equation*}
where
\begin{equation*}\label{kkk}
D(\tau):=2\frac{\pi^{\frac{N-1}{2}}}{\Gamma(\frac{N-1}{2})}\int_0^\pi
\frac{\sin^{N-2}(\theta)}{(1-2\sigma \cos
(\theta)+\tau^2)^{\frac{N+2s}{2}}}d\theta.
\end{equation*}
Considering the behavior of $D$ near from $0$, $1$ and at
$\infty$ (see \cite{FV}), we obtain that
$$
\int_{\frac 87}^{\infty} \tau^{N-\g-1}D(\tau)d\tau\le C,
$$
and therefore we conclude that
\begin{equation}\label{eq2}
\int_{\ren \setminus
B_{8r}}\int_{B_{8r}}(v(x)-v(y))\frac{\psi^{2}(x)}{v(x)}d\nu\le
Cr^{N-2s-2\g}.
\end{equation}
Moreover
\begin{equation}\begin{split}\label{eq3}
\dyle \int_{B_{8r}}\int_{B_{8r}}
(v(x)-v(y))&\left(\frac{\psi^{2}(x)}{v(x)}-\frac{\psi^{2}(y)}{v(y)}\right)
d\nu\\
&=\dyle \int_{B_{6r}}\int_{B_{6r}}
(v(x)-v(y))\left(\frac{\psi^2(x)}{v(x)}-\frac{\psi^2(y)}{v(y)}\right) d\nu \\
&\;\;\;+ \dyle
\iint_{B_{8r}\times B_{8r} \setminus B_{6r}\times
B_{6r}}(v(x)-v(y))\left(\frac{\psi^{2}(x)}{v(x)}-\frac{\psi^{2}(y)}{v(y)}\right)d\nu\\
&\dyle \le \int_{B_{6r}}\int_{B_{6r}}
(v(x)-v(y))\left(\frac{1}{v(x)}-\frac{1}{v(y)}\right) d\nu+Cr^{N-2s-2\g},
\end{split}\end{equation}
where the last inequality follows  as a consequence of that $\psi\equiv 1$ in $B_{6r}$ and that the integral in $B_{8r}\times B_{8r} \setminus B_{6r}\times
B_{6r}$ can be estimated in the same way as \eqref{eq2}.

Furthermore, from \cite[Proof of Lemma 1.3]{CKP1},  there exist $C_1, C_2>0$ such that
\begin{equation}\label{eqlog}
(v(x)-v(y))\left(\frac{\psi^{2}(x)}{v(x)}-\frac{\psi^{2}(y)}{v(y)}\right)\le
-C_1(\log(v(x))-\log(v(y)))^2 \psi^2(y)+C_2(\psi(x)- \psi(y))^2.
\end{equation}
Hence from \eqref{eqlog} we deduce
\begin{equation}\label{eq4}
\int_{B_{6r}}\int_{B_{6r}}
(v(x)-v(y))\left(\frac{1}{v(x)}-\frac{1}{v(y)}\right) d\nu\leq -C_1\int_{B_{6r}}\int_{B_{6r}}(\log(v(x))-\log(v(y)))^2 \,d\nu,
\end{equation}
and thus, putting together \eqref{eq1}, \eqref{eq2}, \eqref{eq3} and \eqref{eq4}, it follows that
\begin{equation}\label{elli3}
\int_{B_{6r}}\int_{B_{6r}} (\log(v(x))-\log(v(y)))^2d\nu\le
Cr^{N-2s-2\g}.
\end{equation}
Let $\delta\in (0,1/4)$. We set $w(x):=\min\{\log\big(\dfrac{1}{2\delta}),
\log(\dfrac{k}{v}\big)\}_+$,and hence, since $w$ is a truncation of $\log\left(\dfrac{k}{v}\right)$, from \eqref{elli3} we obtain that
\begin{equation*}\label{elli4}
\int_{B_{6r}}\int_{B_{6r}} (w(x)-w(y))^2d\nu\le Cr^{N-2s-2\g}.
\end{equation*}
Call
$$\langle
w\rangle_{B_{6r}}:=\dfrac{1}{|B_{6r}|_{d\mu}}\int_{B_{6r}}w(x)d\mu.$$
Thus, using H\"older and Poincaré-Wirtinger inequalities, Theorem \ref{PW},
\begin{equation}\label{eq6}
\int_{B_{6r}}|w(x)-\langle w\rangle_{B_{6r}}|d\mu\le
C|B_{6r}|_{d\mu}.
\end{equation}
Notice that $\{x\in\Omega:\;w(x)=0\}=\{x\in\Omega:\;v(x)\ge k\}$, and then from \eqref{elli1} we have
\begin{equation*}\label{elli111}
|B_{6r}\cap \{w=0\}|_{d\mu}\geq
\dfrac{\sigma}{6^{N-2\g}}|B_{6r}|_{d\mu}.
\end{equation*}
As a consequence of this, it can be seen that
$$
|B_{6r}\cap \{w=\log\Big(\frac{1}{2\delta}\Big)\}|_{d\mu}\le
\dfrac{6^{N-2\g}}{\sigma \log(\frac{1}{2\delta})}\int_{B_{6r}}|w(x)-\langle
w\rangle_{B_{6r}}|d\mu,
$$
and hence, we conclude the result by applying \eqref{eq6} and the fact that
$$
\{B_{6r}\cap \{v\leq 2\delta k\}\}=\{B_{6r}\cap
\{w=\log\Big(\frac{1}{2\delta}\Big)\}\}.
$$
\end{proof}

As a consequence we have the next estimate on $\inf_{B_{4r}} v$.
\begin{Lemma}\label{lema2}
Assume that the hypotheses of Lemma \ref{lema1} are satisfied.
Then, there exists $\delta \in (0,\frac{1}{2})$,
depending only on $N$, $s$, $\sigma$ and $\g$, such that
\begin{equation}\label{estim2} \inf_{B_{4r}} v\ge \delta k.
\end{equation}
\end{Lemma}
\begin{proof} We set $w:=(l-v)_{+}$ where $l\in (\delta k,2\delta k)$ and let
$\psi\in \mathcal{C}^\infty_0(B_\rho)$ with $r\le \rho<6r$.

Using $w\psi^2$ as a test function in \eqref{peso1} and using similar arguments to those in
\cite[Lemma 3.2]{CKP}, we reach that
\begin{equation}\label{last}
\begin{split}
\dyle \int_{B_{\rho}}\int_{B_\rho} &(w(x)\psi(x)-w(y)\psi(y))^2
d\nu \\
\le & \,C_1\dyle
\int_{B_{\rho}}\int_{B_\rho}\max\{w(x),w(y)\}^2(\psi(x)-\psi(y))^2d\nu \\
&+
l^2|B_\rho\cap \{v<l\}|_{d\mu}\times \sup_{\{x\in
\text{supp}(\psi)\}}\int_{\ren\backslash
B_{\rho}}\dfrac{dy}{|x-y|^{N+2s}}.
\end{split}
\end{equation}
We define now the sequences $\{l_j\}_{j\in\mathbb{N}}$,
$\{\rho_j\}_{j\in\mathbb{N}}$ and
$\{\bar{\rho_j}\}_{j\in\mathbb{N}}$ by setting
$$
l_j:=\delta k+2^{-j-1}\delta k,\:\rho_j:=4r+2^{1-j}r,\:
\bar{\rho}_j:=\frac{\rho_j+\rho_{j+1}}{2}.
$$
Likewise, let us denote
$$w_j:=(l_j-v)_+,\qquad B_j:=B_{\rho_j},$$
and let $\psi_j\in \mathcal{C}_0^\infty(B_{\bar{\rho_j}})$ be such
that $0\leq \psi\leq 1$, $\psi\equiv 1$ in $B_{j+1}$ and
$|\nabla\psi_j|\leq 2^{j+3}/r$.

Using the Sobolev inequality stated in Proposition \ref{Sobolev00}
we obtain that
$$
C(N,s,\g)\Big(\dint\limits_{B_j}
\dfrac{|w_j\psi_j|^{2^*_{s}}}{|x|^{\g
2^*_s}}\,dx\Big)^{\frac{2}{2^*_{s}}}\le
\dint_{B_j}\dint_{B_j}(w_j(x)\psi_j(x)-w_j(y)\psi_j(y))^2d\nu.
$$
Hence, using the facts that
$$w_j\psi_j\ge (l_j-l_{j+1})\hbox{ in }B_{j+1}\cap \{v<l_{j+1}\},$$
and
$$|x|^{-2^*_s\g}\ge \bar{C}r^{-(2^*_s-2)\g}|x|^{-2\g}\hbox{ in }B_j,$$
with $\bar{C}$ independent of $j$, it follows that
$$
\Big(\dint\limits_{B_j} \dfrac{|w_j\psi_j|^{2^*_{s}}}{|x|^{\g
2^*_s}}\,dx\Big)^{\frac{2}{2^*_{s}}}\ge \frac{C}{r^{(2^*_s-2)\g
}}(l_j-l_{j+1})^2|B_{j+1} \cap \{v<l_{j+1}\}|^{\frac{2}{2^*_{s}}}_{d\mu}.
$$
Since $|B_{j+1}|_{d\mu}= C r^{N-2\gamma}$, then
$$\dfrac{r^{(2^*_s-2)\g
}}{(|B_{j+1}|_{d\mu})^{\frac{2}{2^*}}}=C
r^{-(N-2s-2\g)}r^{4s\gamma(\frac{1}{N-2s}-\frac{1}{N})}\le
C(N,s,\gamma,\Omega)r^{-(N-2s-2\g)}.$$
Hence we conclude that
\begin{equation*}\begin{split}
(l_j-l_{j+1})^2&\Big(\dfrac{|B_{j+1} \cap
\{v<j+1\}|_{d\mu}}{|B_{j+1}|_{d\mu}}\Big)^{\frac{2}{2_{s}^*}}\\
&\le
C(N,s,\g,\Omega)r^{-(N-2s-2\g)}\dint_{B_j}\dint_{B_j}(w_j(x)\psi_j(x)-w_j(y)\psi_j(y))^2d\nu.
\end{split}\end{equation*}
Applying \eqref{last} to $w_j$, we conclude that
\begin{equation}\label{last1}
\begin{split}
(l_j-l_{j+1})^2&\Big(\dfrac{|B_{j+1} \cap
\{v<j+1\}|_{d\mu}}{|B_{j+1}|_{d\mu}}\Big)^{\frac{2}{2_{s}^*}}\\
\le & \,\dfrac{C(N,s,\g)}{r^{(N-2s-2\g)}}\left(C_1\dyle
\int_{B_{j}}\int_{B_j}\max\{w_j(x),w_j(y)\}^2(\psi_j(x)-\psi_j(y))^2d\nu \right.\\
&\left.+
l^2_j|B_j\cap \{v<l_j\}|_{d\mu}\, \sup_{\{x\in
\text{supp}(\psi_j)\}}\int_{\ren\backslash
B_{j}}\dfrac{dy}{|x-y|^{N+2s}}\right).
\end{split}
\end{equation}
We have
\begin{equation}\label{lastt1}
\begin{split}
\dyle \int_{B_{j}}\int_{B_j}\max&\{w_j(x),w_j(y)\}^2(\psi_j(x)-\psi_j(y))^2d\nu\\
& \le
\dyle l_{j}^2\|\nabla\psi_j\|_{L^\infty(B_j)}^2\int_{B_{j}\cap
\{v<l_j\}}\dfrac{dx}{|x|^\g}\int_{B_j}\dfrac{|x-y|^{2-2s}}{|x-y|^N}\dfrac{dy}{|y|^\g}\\
&\le \dyle C2^{2j}l_{j}^2
r^{-2s}\int_{B_{j}\cap \{v<l_j\}}\dfrac{dx}{|x|^{2\g}}=C2^{2j}l_{j}^2 r^{-2s}|B_{j}\cap
\{v<l_j\}|_{d\mu}.
\end{split}\end{equation}
Now, estimating the
term
$$\sup_{\{x\in
\text{supp}(\psi_j)\}}\int_{\ren\backslash
B_{j}}\dfrac{dy}{|x-y|^{N+2s}}$$ as in \cite[Lemma 3.2]{CKP},
and considering \eqref{last1} and
\eqref{lastt1}, we obtain that
\begin{equation*}\begin{split}
(l_j-l_{j+1})^2\Big(\dfrac{|B_{j+1} \cap
\{v<j+1\}|_{d\mu}}{|B_{j+1}|_{d\mu}}\Big)^{\frac{2}{2_{s}^*}}&\le\
2^{j(2+2s+N)}l_{j}^2
\dfrac{C(N,s,\gamma)}{r^{(N-2s-2\g)}}r^{-2s}|B_{j}\cap
\{v<l_j\}|_{d\mu}\\
&\le
\tilde{C}2^{j(2+2s+N)}l_{j}^2\dfrac{|B_{j}
\cap \{v<j\}|_{d\mu}}{|B_{j}|_{d\mu}}
\end{split}\end{equation*}
where $\tilde{C}=\tilde{C}(N,s,\gamma)$ but independent
of $j$ and $r$.

Defining $A_j:=\dfrac{|B_{j} \cap \{v<j\}|_{d\mu}}{|B_{j}|_{d\mu}}$
and following as in \cite{CKP}, we get the desired result.
\end{proof}

Now, we need to obtain a kind of \emph{reverse Hölder inequality}
for $v$.
\begin{Lemma}\label{dos}
Let $r>0$ such that $B_{3r/2}\subset\Omega$ and suppose that $v$ is a supersolution to \eqref{peso1}. Then, for every
$0<\alpha_1<\alpha_2<\frac{N}{N-2s}$, we have
\begin{equation}\label{est3}
\left(\dfrac{1}{|B_{r}|_{d\mu}}\dint\limits_{B_{r}}v^{\alpha_2}\,d\mu\right)^{\frac{1}{\alpha_2}}
\le C \left(\dfrac{1}{|B_{3r/2}|_{d\mu}}\dint\limits_{B_{3r/2}}v^{\alpha_1}\,d\mu\right)^{\frac{1}{\alpha_1}},
\end{equation}
with $C=C(N,s,\gamma,\alpha_1,\alpha_2)>0$.
\end{Lemma}
\begin{proof}
Let $q\in (1,2)$ and $d>0$. Set $\tilde{v}:=(v+d)$, and assume that
$\psi\in \mathcal{C}^\infty_0(\Omega)$ is such that
$\text{supp}(\psi)\subset B_{\tau r}, \psi=1$ in $B_{\tau' r}$ and
$|\nabla \psi|\le \frac{C}{(\tau-\tau')r}$ where $\frac 12 \le
\tau'<\tau<\frac 32$. Then using $\tilde{v}^{1-q}\psi^2$ as a test
function in \eqref{super}, we obtain that
\begin{eqnarray*}
& 0\dyle \le \int_{B_{{\tau r}}}\int_{B_{{\tau r}}}
(\tilde{v}(x)-\tilde{v}(y))\left(\frac{\psi^{2}(x)}{\tilde{v}^{q-1}(x)}-\frac{\psi^{2}(y)}{\tilde{v}^{q-1}(y)}\right)
d\nu + 2\int_{\ren \setminus
B_{{\tau r}}}\int_{B_{\tau r}}(\tilde{v}(x)-\tilde{v}(y))\frac{\psi^{2}(x)}{\tilde{v}^{q-1}(x)}d\nu.
\end{eqnarray*}
Since $|x|<|y|$ in $B_{\tau r}\times (\R^N\setminus B_{\tau r})$, and using the positivity of $\tilde{v}$ it follows
$$
\int_{\ren \setminus
B_{{\tau r}}}\int_{B_{{\tau r}}}(\tilde{v}(x)-\tilde{v}(y))\frac{\psi^{2}(x)}{{\tilde{v}^{q-1}(x)}}d\nu\le
\left(\dint\limits_{B_{{\tau r}}}\tilde{v}^{2-q}\psi^2\,d\mu \right)\left(
\sup_{\{x\in \text{supp}(\psi)\}}\int_{\ren\backslash
B_{{\tau r}}}\dfrac{dy}{|x-y|^{N+2s}}\right).
$$
Furthermore, by the pointwise inequality of Lemma 3.3-(i) in \cite{FK},  there exist positive constants $C_1$ and $C_2$,
depending on $q$, such that
\begin{equation*}\begin{split}
\int_{B_{\tau r}}\int_{B_{\tau r}}
(\tilde{v}(x)-\tilde{v}(y))&\left(\frac{\psi^{2}(x)}{\tilde{v}^{q-1}(x)}-\frac{\psi^{2}(y)}{\tilde{v}^{q-1}(y)}\right)
d\nu\\
\le &-C_1\int_{B_{\tau r}}\int_{B_{\tau r}}
(\tilde{v}^{\frac{2-q}{2}}(x)\psi(x)-\tilde{v}^{\frac{2-q}{2}}(y)\psi(y))^2d\nu
\\
&+C_2\int_{B_{\tau r}}\int_{B_{\tau r}}
((\tilde{v}^{{2-q}}(x)+\tilde{v}^{2-q}(y))(\psi(x)-\psi(y))^2
d\nu.
\end{split}\end{equation*}
By symmetry we have
\begin{equation*}\begin{split}
\int_{B_{\tau r}}\int_{B_{\tau
r}} ((\tilde{v}^{{2-q}}(x)+\tilde{v}^{2-q}(y))(\psi(x)-\psi(y))^2
d\nu=2\int_{B_{\tau
r}}\int_{B_{\tau r}}
(\tilde{v}^{{2-q}}(x)(\psi(x)-\psi(y))^2 d\nu
\end{split}\end{equation*}
and proceeding as in \cite[Lemma 4.6]{AMPP} we obtain
$$
\int_{B_{\tau r}}\int_{B_{\tau r}}
((\tilde{v}^{{2-q}}(x)+\tilde{v}^{2-q}(y))(\psi(x)-\psi(y))^2
d\nu\le \dfrac{Cr^{-2s}}{(\tau-\tau')^2}\dint\limits_{B_{\tau r}}\tilde{v}^{2-q}\,d\mu.
$$
Since
$$
\sup_{\{x\in \text{Supp}(\psi)\}}\int_{\ren\backslash
B_{{\tau r}}}\dfrac{dy}{|x-y|^{N+2s}}\le Cr^{-2s},
$$
then combining the estimates above we reach that
$$
\int_{B_{{\tau r}}}\int_{B_{{\tau r}}}
{(\tilde{v}^{\frac{2-q}{2}}(x)\psi(x)-\tilde{v}^{\frac{2-q}{2}}(y)\psi(y))^2}d\nu
\le \dfrac{Cr^{-2s}}{(\tau-\tau')^2}\dint\limits_{B_{\tau
r}}\tilde{v}^{2-q}\,d\mu.
$$

Hence, from the previous inequality and the Sobolev inequality in Proposition \ref{Sobolev00}, we get
\begin{equation*}
\begin{split}
\left(\dfrac{1}{|B_{\tau' r}|_{d\mu}}\dint\limits_{B_{\tau'
r}}\tilde{v}^{\frac{(2-q)N}{N-2s}}\,d\mu\right)^{\frac{N-2s}{N}}&\leq
\left(\dfrac{1}{|B_{\tau' r}|_{d\mu}}\dint\limits_{B_{\tau
r}}(\tilde{v}^{\frac{2-q}{2}}\psi)^{2^*_s}\,d\mu\right)^{\frac{N-2s}{N}} \\
&\le
\dfrac{C}{|B_{\tau r}|_{d\mu}(\tau-\tau')^2}\dint\limits_{B_{\tau
r}}\tilde{v}^{2-q}\,d\mu.
\end{split}
\end{equation*}
Since $q\in (1,2)$ is arbitrary and $\frac{N}{N-2s}>1$ by using Hölder inequality we obtain the estimate
\eqref{est3} for $\tilde{v}=v+d$ with $\alpha_1$ and  $\alpha_2$ in the hypotheses.
Finally letting $d\to 0$ and by the \textit{Monotone Convergence Theorem}  we conclude.
\end{proof}

In order to obtain the weak Harnack inequality, we need to prove the following estimate.
\begin{Lemma}\label{tres}
Let $r>0$ such that $B_r\subset\Omega$. Assume that $v$ is a supersolution to \eqref{peso1}.Then, there
exists a constant $\eta\in (0,1)$ depending only on $N$, $s$ and $\gamma$ such that
\begin{equation*}\label{est31}
\dyle\Big(\dfrac{1}{|B_r|_{d\mu}}\int_{B_r}v^{\eta}d\mu\Big)^{\frac{1}{\eta}}\le
C\inf_{B_r} v.
\end{equation*}
\end{Lemma}
To prove Lemma \ref{tres} (see \cite{KS} and \cite[Lemma 4.1]{CKP}) we need the next covering result in the
spirit of Krylov-Safonov theory.
Notice that  we are working with a doubling measure on bounded domains of $\mathbb{R}^N$.
\begin{Lemma}\label{cover}
Assume that $E\subset B_r(x_0)$ is a measurable set. For
$\bar{\delta}\in (0,1)$, we define
$$
[E]_{\bar{\delta}}:= \bigcup_{\rho>0}\{B_{3\rho}(x)\cap
B_r(x_0), x\in B_r(x_0): |E\cap B_{3\rho}(x)|_{d\mu}>\bar{\delta}|
B_{\rho}(x)|_{d\mu}\}.
$$
Then, either
\begin{enumerate}
\item $
|[E]_{\bar{\delta}}|_{d\mu}\ge\frac{\tilde{C}}{\bar{\delta}}|E|_{d\mu}$,
or \item $ [E]_{\bar{\delta}}=B_r(x_0)$,
\end{enumerate}
where $\tilde{C}$ depends only on $N$, $s$ and $\gamma$.
\end{Lemma}
\begin{pfn}{Lemma \ref{tres}}
Notice that, for any $\eta>0$,
\begin{equation}\label{rep}
\dyle\dfrac{1}{|B_r|_{d\mu}}\int_{B_r}v^{\eta}d\mu=\eta
\int_0^\infty
t^{\eta-1}\dfrac{|B_r\cap\{v>t\}|_{d\mu}}{|B_r|_{d\mu}}dt.
\end{equation}
Then, for $t>0$ and $i\in \mathbb{N}$, we set $ A^i_t:=\{x\in B_r:
v(x)>t\delta^i\}$ where $\delta $ is given by Lemma \ref{lema2}.
Notice that $A^{i-1}_t\subset A^i_t$.

Let $\rho>0$ and $x\in B_r$ such that $B_{3\rho}(x)\cap B_r\subset
[A^{i-1}_t]_{\bar{\delta}}$. Thus,
$$
|A^{i-1}_t\cap
B_{3\rho}(x)|_{d\mu}>\bar{\delta}|B_\rho|_{d\mu}=\frac{\bar{\delta}}{3^{N-2\g}}|B_{3\rho}|_{d\mu}.
$$
Hence, using Lemma \ref{lema2}, we reach that
$$
v(x)>\delta (t\delta^{i-1})=t\delta^i  \mbox{  for all }x\in B_r,
$$
and therefore $[A^{i-1}_t]_{\bar{\delta}}\subset A^{i}_t$, being
$[A^{i-1}_t]_{\bar{\delta}}$ as in Lemma \ref{cover}. This
fact, together with Lemma \ref{cover}, allows us to deduce that
\begin{equation}\label{alt1}A^{i}_t=B_r\;\hbox{ or }\;|A^{i}_t|_{d\mu}\ge
\frac{\tilde{C}}{\bar{\delta}}|A^{i-1}_t|_{d\mu}.
\end{equation}
Thus, if for some $m\in \mathbb{N}$ we have
\begin{equation}\label{nesr}
|A^0_t|_{d\mu}>\left(\frac{\bar{\delta}}{\tilde{C}}\right)^m|B_r|_{d\mu},
\end{equation}
then $A^{m}_t=B_r$. If not, it follows from
\eqref{alt1} that $$|A^{m}_t|_{d\mu}\ge
\frac{\tilde{C}}{\bar{\delta}}|A^{m-1}_t|_{d\mu}.$$
Since $A^{i-1}_t\subset A^{m}_t\subsetneq B_r$ for all $i\le m$,
the second point of the alternative \eqref{alt1} holds for
$A^{i-1}_t$ and then
$$|A^{m-1}_t|_{d\mu}\ge
\frac{\tilde{C}}{{\bar{\delta}}}|A^{m-2}_t|_{d\mu}....\ge
\left(\frac{\tilde{C}}{\bar{\delta}}\right)^{m-1}|A^{0}_t|_{d\mu}>\left(\frac{\tilde{C}}{\bar{\delta}}\right)^{-1}|B_r|_{d\mu}.$$
Thus $|A^{m}_t|_{d\mu}>|B_r|_{d\mu}$, a contradiction with the
fact that $A^{m}_t\subsetneq B_r$. Hence $A^{m}_t=B_r$.

It is clear that \eqref{nesr} holds if
\begin{equation}\label{mmm}
m>\frac{1}{\log(\frac{\bar{\delta}}{\tilde{C}})}\log\left(\frac{|A^{0}_t|_{d\mu}}{|B_r|_{d\mu}}\right),\end{equation}
and consequently, fixing $m$ to be the smallest
integer such that \eqref{mmm} holds, then $m\ge 1$ and
$$
0\le m-1\le
\frac{1}{\log(\frac{\bar{\delta}}{\tilde{C}})}\log\left(\frac{|A^{0}_t|_{d\mu}}{|B_r|_{d\mu}}\right).
$$
Thus, using the fact that $\delta\in (0,\frac 12)$, it can be
checked that
$$
\inf_{B_r}v>t\delta^m=\geq t\delta\left(\frac{|A^{0}_t|_{d\mu}}{|B_r|_{d\mu}}\right)^{\frac{1}{\beta}},
$$
with
$\beta:=\frac{\log(\frac{\bar{\delta}}{\tilde{C}})}{\log(\delta)}$.

Set now $\xi:=\inf_{B_r}v$. Then,
$$
\dfrac{|B_r\cap\{v>t\}|_{d\mu}}{|B_r|_{d\mu}}=\frac{|A^{0}_t|_{d\mu}}{|B_r|_{d\mu}}\le
\tilde{C}\delta^{-\beta}t^{-\beta}\xi^\beta.
$$
Going back to \eqref{rep}, we have
\begin{equation*}\begin{split}
\dyle\dfrac{1}{|B_r|_{d\mu}}\int_{B_r}v^{\eta}d\mu&\le
\eta\int_0^at^{\eta-1}dt +\eta \tilde{C}\int_a^\infty
t^{\eta-1}\delta^{-\beta}t^{-\beta}\xi^\beta dt\\
&= a^\eta-\eta\tilde{C}\delta^{-\beta}\xi^{\beta}\frac{a^{\eta-\beta}}{\eta-\beta}.
\end{split}\end{equation*}
Choosing $a:=\xi$ and $\eta:=\frac{\beta}{2}$, we reach the
result.
\end{pfn}

After this result, we can already prove the weighted weak Harnack inequality.
\medskip

\begin{pfn}{Theorem \ref{harnack}}
Using Lemma \ref{tres} we
obtain that
$$
\dyle\Big(\dfrac{1}{|B_r|_{d\mu}}\int_{B_r}v^{\eta}d\mu\Big)^{\frac{1}{\eta}}\le
C\inf_{B_r} v
$$
for some $\eta\in (0,1)$. Fixing $1\le q<\frac{N}{N-2s}$, by
Lemma \ref{dos} with $\alpha_1=\eta$ and $\alpha_2=q$,  it follows that
\begin{equation*}\label{est33}
\left(\dfrac{1}{|B_{r}|_{d\mu}}\dint\limits_{B_{
r}}v^{q}\,d\mu\right)^{\frac{1}{q}} \le C \left(\dfrac{1}{|B_{\frac
32 r}|_{d\mu}}\dint\limits_{B_{\frac32
r}}v^{\eta}\,d\mu\right)^{\frac{1}{\eta}}.
\end{equation*}
Hence
$$
\Big(\dfrac{1}{|B_{r}|_{d\mu}}\dint\limits_{B_{
r}}v^{q}\,d\mu\Big)^{\frac{1}{q}} \le C\inf_{B_{\frac 32 r}} v
$$
and we conclude.
\end{pfn}

As a consequence of
the previous Harnack inequality, we get much information about the behavior of the supersolutions to \eqref{prob} around the origin. In particular, we see that any of them must be unbounded, even if  $f\in L^\infty(\Omega)$.
\begin{Lemma}\label{Lm:singularity_elip}
Let $\l\le
   \Lambda_{N,s}$. Assume that $u$ is a nonnegative function defined in $\Omega$ such
that $u\not\equiv 0$,\, $u\in L^1(\Omega)$,
$\dfrac{u}{|x|^{2s}}\in L^1(\Omega)$ and $u\geq 0\inn
\mathbb{R}^{N}\setminus\Omega$. If $u$ satisfies ${(-\Delta)^s }
u-\l\dfrac{u}{|x|^{2s}}\geq 0$ {in the weak sense} in $\Omega$,
then there exists $\delta>0$, and a constant  $C=C(N,\delta,\gamma)$ such
that for each ball $B_{\delta}(0)\subset\subset \Omega$,
$$u\geq C |x|^{-\gamma}\hbox{ in }  B_{\delta}(0),$$ where $\gamma$ is defined in
\eqref{lambda}. In particular, for $\delta$ conveniently small, we can
assume that $u>1$ in $B_{\delta}(0).$
\end{Lemma}
\begin{proof}
Considering $v:=|x|^{\gamma}u$, then
$v\gneqq 0$ and it satisfies $L_{\gamma} v\ge 0$, with $L_{\gamma}$ defined in \eqref{Ltilde0}. Hence
using the weak Harnack inequality in Theorem \ref{harnack}, we
conclude that $\inf_{B_r(0)}v\ge C$. Thus $u(x)\ge C|x|^{-\gamma}$
in $B_r(0)$ and the result follows.
\end{proof}

\section{Optimal summability in the presence of Hardy
potential}\label{s4}
In this section we analyze the question of
the optimal summability of the solution to the problem
\begin{equation}\label{mainfirst}
\left\{
\begin{array}{rcl}
(-\Delta)^s u -\lambda\dfrac{u}{|x|^{2s}} &=& f \inn\Omega,\\
u&=&0\inn \mathbb{R}^{N}\setminus\Omega,
\end{array}\right.
\end{equation}
with $0<\lambda< \Lambda_{N,s}$.
\subsection{Regularity of energy solution}

Along this subsection we will assume that $f\in L^m(\Omega)$ with
$m\ge \frac{2N}{N+2s}$, and thus the solution $u$ will belong to $H_0^s(\Omega)$.

In particular, it is  known that in the classical case, i.e. when $\lambda=0$, if $m>\frac{N}{2s}$, then $u\in L^\infty(\Omega)$.
However, as a consequence of Lemma \ref{Lm:singularity_elip}, this feature is no longer true for $\lambda>0$, and actually $u(x)\ge
C|x|^{-\gamma}$ in a neighborhood of the origin.
Hence, a natural question here is whether this rate is exactly the rate of growth of $u$, and the answer is yes for regular data, as the following theorem shows.

\begin{Theorem} \label{singular}
Let $f\in L^{m}(\Omega), m>\frac{N}{2s}$. Let consider $u\in
H_{0}^{s}(\Omega)$ the unique energy solution to problem
\eqref{mainfirst}, with $\lambda\le   \Lambda_{N,s}$, then
$u(x)\leq C |x|^{-\gamma}\inn\mathbb{R}^{N}$.
\end{Theorem}
\begin{proof}
Since the problem is linear, without loss of generality we can assume $f\ge 0$.
Defining $v(x):= |x|^{\gamma} u(x)$, it can be checked that it solves
\begin{equation}\label{weightedProblem}
\left\{
\begin{array}{rcl}
L_\gamma (v) &=& |x|^{-\gamma}f \inn\Omega,\\
v&=&0\inn \mathbb{R}^{N}\setminus\Omega,
\end{array}\right.
\end{equation}
where the operator $L_\gamma (v)$ was defined by \eqref{Ltilde0}.
Consider now $G_{k}(v(x))$, specified in \eqref{Tk}, with $k>0$ as test function in {\eqref{weightedProblem}}. Hence,
$$
\frac{a_{N,s}}{2}\iint_{D_\Omega} \dfrac{{\big( v (x) - v (y) \big)\big(G_k (v (x))- G_k(v (y)) \big) } }{|x-y|^{N+2s}} \dfrac{dx}{|x|^{\gamma}} \dfrac{dy}{|y|^{\gamma}}=\int_\Omega{|x|^{-\gamma}fG_k(v)\,dx} \,.
$$
Since for any  $\sigma \in \mathbb{R}$, $ \sigma=T_k (\sigma)+G_k (\sigma)$, then
$$
\begin{array}{c}
\big( v (x) - v (y) \big)\big(G_k (v (x))- G_k(v (y)) \big)
\\
= \big(G_k (v (x))- G_k(v (y)) \big)^2  + \big(T_k (v (x))- T_k(v (y)) \big)\big(G_k (v (x))- G_k(v (y)) \big).
\end{array}$$
Moreover, by \cite[Lemma 2.5]{LPPS}, we know that
$$
\big(T_k (v (x))- T_k(v (y)) \big)\big(G_k (v (x))- G_k(v (y)) \big) \ge 0,
$$
and therefore
$$
\frac{a_{N,s}}{2}
\iint_{D_\Omega}\frac{|G_{k}(v(x))-G_{k}(v(y))|^2}{|x-y|^{N+2s}}\frac{dx}{|x|^{\gamma}}\,
\frac{dy}{|y|^{\gamma}}\leq  \dint_{\Omega}{f}
\,G_{k}(v(x))\,\dfrac{dx}{|x|^{\gamma}}.
$$
Let us denote $A_k :=\{ x \in \Omega\,: \, v(x)  \geq k\}$.
Applying the weighted Sobolev inequality (Proposition \ref{Sobolev}) in the left hand side we obtain,
$$
\mathcal{S} \| G_k (v   ) \|_{L^{2^*_s} (\Omega,|x|^{-\gamma}\,dx)}^2 \leq   \int_{A_{k}}\,{f}\,G_k (v  (x)) \frac{dx}{|x|^\gamma}
$$
and using H\"older's inequality in the right hand side,
$$
 \bigg|\int_{A_{k}}\,f\,G_k (v  (x)) \frac{dx}{|x|^\gamma}\bigg|
\leq \|f \|_{L^{m}(\Omega)} \|G_k (v  )\|_{L^{2^*_s}(\Omega,{|x|^{-\gamma}\,dx})} |A_k|^{1-\frac1{2^*_s}-\frac1{m}}.
$$
Thus we have that
$$
\| G_k (v   ) \|_{L^{2^*_s} (\Omega,|x|^{-\gamma}\,dx)} \leq \mathcal{S}^{-1} \|f \|_{L^{m}(\Omega)}  |A_k|^{1-\frac1{2^*_s}-\frac1{m}} \,.
$$
{On the other hand, since $\Omega$ is bounded, there exists a constant $c>0$ such that $\| G_k (v   ) \|_{L^{2^*_s} (\Omega,|x|^{-\gamma}\,dx)}\geq c\| G_k (v   ) \|_{L^{2^*_s} (\Omega)}$.} Moreover,  for any $z>k$, we have that $A_z \subset A_k$ and $G_k (s) \chi_{A_z} \geq (z-k)$ for every $s\in\R$ and thus
$$
(z-k) |A_z|^{\frac1{2^*_s}} \leq \frac{1}{{c \mathcal{S}}} \|f \|_{L^{m}(\Omega)}  |A_k|^{1-\frac1{2^*_s}-\frac1{m}}.
$$
Manipulating the above inequality we deduce that
$$
 |A_z| \leq \frac{\|f \|^{2^*_s}_{L^{m}(\Omega)}  |A_k|^{2^*_s (1-\frac1{2^*_s}-\frac1{m}) } }{{(c  \mathcal{S} )}^{2^*_s} (z-k)^{2^*_s}}\,.
$$
Hence we apply Lemma \ref{st} with the choice $\psi (s):= |A_s|$, using that
$$
2^*_s \Big(1-\frac1{2^*_s}-\frac1{m}\Big)>1,
$$
since $m>\frac{N}{2s}$. Consequently there exists $k_0$ such that $\psi (k) \equiv 0 $ for any $k \geq k_0$ and thus
$$\displaystyle\text{ ess }\sup_{\Omega} v  \leq k_0.$$
\end{proof}
We next try to obtain under which conditions of $\lambda$ the Calderón-Zygmund summability holds for the rest of the Lebesgue  spaces
contained in the dual of $H^s_0(\Omega)$, that is, $f\in L^m(\Omega)$ where $\frac{2N}{N+2s}\le m< \frac{N}{2s}$.
Then we have the following result.
\begin{Theorem} \label{CZ}
Let $f$ be a positive function  $f\in L^ {m}(\Omega)$, with $\frac{2N}{N+2s}\le m < \frac{N}{2s}$. If
\begin{equation}\label{condi}
\lambda<\Lambda_{N,s}\dfrac{4N(m-1)(N-2ms)}{m^2(N-2s)^2},
\end{equation}
then there exists a constant $c=c(N,m,s)>0$ such that the unique energy
solution of problem \eqref{mainfirst} verifies
\begin{equation}\label{CZbou1}
\|{u}\|_{L^{m^{**}_s}(\Omega)} \leq c\,\|{f}\|_{L^ {m}(\Omega)}
\qquad \mbox{ where } \qquad m^{**}_s=\frac{mN}{N-2ms}\, .
\end{equation}
\end{Theorem}

\begin{proof} Since $f\in H^{-s}(\Omega)$, the existence and uniqueness of an energy solution can be proved by means of a direct abstract Hilbert space approach.

To study the regularity of the solution, for every $k\in\mathbb{N}$, we consider $u_k\in  L^{\infty}(\ren)\cap H_{0}^{s}(\Omega)$,
the solution to the following approximated problem
\begin{equation}\label{eq:pk}
\left\{\begin{array}{rcll}
(-\Delta)^s u_k-\lambda \dfrac{u_{k-1}}{|x|^{2s}+\frac 1k}&=&f_k(x) &\hbox{  in } \Omega,\\
u_k&=&0 &\hbox{  in } \mathbb{R}^N\setminus\Omega,\\
\end{array}
\right.
\end{equation}
where $f_k(x):=\min\{ f(x), k\}$ and $u_{0}=0$.

In this way we  obtain the following properties: $(i)$ $\{u_k\}$ is an increasing sequence;
 $(ii)$ each $u_k$ is bounded, and
$(iii)$ $u_k\to u$, the unique solution to problem \eqref{prob}, in $L^p(\Omega)$, for every $1\leq p \le 2^*_s$.

Define $\beta=\dfrac{2^*_s}{2m'-2^*_s}=\dfrac{N(m-1)}{N-2ms}\geq 1$,  that satisfies
$\beta m'=\frac{(\beta+1)}{2}{2^*_s}$.
Since $u_k$ is bounded we can take $u_k^\beta$ as a test function in
\eqref{mainfirst}, obtaining
\begin{eqnarray*}
\frac{a_{N,s}}{2}\iint_{D_\Omega}\dfrac{(u_k(x)-u_k(y))(u_k^{\beta}(x)-u_k^{\beta}(y))}{|x-y|^{N+2s}}\,dx\,dy\le\lambda
\int_{\Omega}\dfrac{u_k^{\beta+1}}{|x|^{2s}}dx+ \int_{\Omega} f_k
u_k^{\beta} dx.
\end{eqnarray*}
By H\"older's inequality,
$$
\int_{\Omega} f_k u_k^{\beta} dx\le
\|f_k\|_{L^m(\Omega)}\Big(\int_{\Omega}u_k^{m'\beta}
dx\Big)^{\frac{1}{m'}}\le
\|f\|_{L^m(\Omega)}\Big(\int_{\Omega}u_k^{\frac{(\beta+1)2^*_s}{2}}
dx\Big)^{\frac{\beta}{\beta+1}\frac{2}{2^*_s}}.
$$
Now, by the algebraic inequality in \eqref{alg0}, we get
$$
(u_k(x)-u_k(y))(u_k^{\beta}(x)-u_k^{\beta}(y))\ge
\frac{4\beta}{(\beta+1)^2}(u_k^{\frac{\beta+1}{2}}(x)-u_k^{\frac{\beta+1}{2}}(y))^2,
$$
and hence, using Hardy's inequality again,  we conclude that
$$
\frac{a_{N,s}}{2}\Big(\frac{4\beta}{(\beta+1)^2}-\frac{\lambda}{\Lambda_{N,s}}\Big)
\iint_{D_\Omega}\dfrac{(u_k^{\frac{\beta+1}{2}}(x)-u_k^{\frac{\beta+1}{2}}(y))^2}{|x-y|^{N+2s}}
\,dx\,dy\le
\|f\|_{L^m(\Omega)}\Big(\int_{\Omega}u_k^{\frac{(\beta+1)2^*_s}{2}}
dx\Big)^{\frac{\beta}{\beta+1}\frac{2}{2^*_s}}.$$
On the other hand, hypothesis \eqref{condi} is equivalent to
$$\Big(\frac{4\beta}{(\beta+1)^2}-\frac{\lambda}{\Lambda_{N,s}}\Big)>0,$$
and thus, by the Sobolev inequality, we reach that
$$
\Big(\int_{\Omega}u_k^{\frac{(\beta+1)2^*_s}{2}}
dx\Big)^{\frac{2}{2^*_s(\beta+1)}}\le C \|f\|_{L^m(\Omega)}.
$$
Furthermore, $\frac{(\beta+1)2^*_s}{2}=\dfrac{mN}{N-2sm}=
m_{s}^{**}$, and therefore passing to the limit we conclude.
\end{proof}

\begin{remark} {Notice that, making $s\rightarrow 1$, the condition over $\lambda$ becomes
$$
\lambda<\dfrac{N(m-1)(N-2m)}{m^2},$$
the curve obtained in \cite{BOP} for  the  local  case. }
\end{remark}
\subsection{About the optimality of the regularity results.}

For simplicity of typing we set
\begin{equation}\label{Jsm}
J_s(m):=\Lambda_{N,s}\dfrac{4N(m-1)(N-2ms)}{m^2(N-2s)^2}.
\end{equation}
In \cite{BOP}, the authors proved that in the local case
condition \eqref{condi} (with $s=1$) is optimal. In particular, they see  that if $\lambda>J_1(m)$ and $\Omega=B_1(0)$, there
exists a suitable radial function $f\in L^m(\Omega)$ such that the
solution $u$ does not belong to $L^{m^{**}}(\Omega)$.

In the nonlocal case the situation is more delicate. Indeed, we will see that in this case the previous example does not provide the optimality of the curve $J_s(m)$. It proves that the $m_s^{**}$ summability does not hold above a curve, that we will call $P_s(m)$, that is in general above of $J_s(m)$. Thus, as far as we know, the optimality of $J_s(m)$ for every $\frac{2N}{N+2s}\leq m <\frac{N}{2s}$ remains open.

In order to define such curve $P_s(m)$, let us first recall  that for all $\lambda\in (0,\Lambda_{N,s}]$ there exists a
unique nonnegative constant $\alpha\in [0,\frac{N-2s}{2})$ such
that
\begin{equation}\label{lambda11}
\lambda=\lambda(\alpha)=\dfrac{2^{2s}\,\Gamma(\frac{N+2s+2\alpha}{4})\Gamma(\frac{N+2s-2\alpha}{4})}{\Gamma(\frac{N-2s+2\alpha}{4})\Gamma(\frac{N-2s-2\alpha}{4})},
\end{equation}
and $\lambda(\alpha)$ is a decreasing function of $\alpha$ (see \cite{AMPP}). Hence, for $m\in [\frac{2N}{N+2s}, \frac{N}{2s})$ fixed, we consider
\begin{equation}\label{Psm}
\alpha_0(m):=\frac{N+2s}{2}-\frac{N}{m}\qquad \hbox{ and }\qquad P_s(m):=\lambda(\alpha_0(m))
\end{equation}
given by \eqref{lambda11}.
Then we have the next result.
\begin{Lemma}\label{opt1}
Assume that  $f(x)=\dfrac{1}{|x|^{\nu}}$ with
$\nu=\dfrac{N-\e}{m}$, for some $\e>0$. Let $\l_1\in
(0,\Lambda_{N,s}]$ be such that $\lambda_1\ge P_s(m)$. If $u$ is
the unique solution of \eqref{mainfirst} with $\lambda=\lambda_1$,
then $u\notin L^{m^{**}_s}(B_1(0))$.
\end{Lemma}
\begin{proof} Notice that $f\in L^m(B_1(0))$.
From \eqref{lambda11} we know that $\l_1=\l(\alpha_1)$ for some
$\alpha_1\in [0,\frac{N-2s}{2})$. Since $\lambda_1\ge
P_s(m)$, using the fact that $\l(\alpha)$ is a decreasing
function we reach that
$$\alpha_1\le
\alpha_0(m)=\frac{N+2s}{2}-\frac{N}{m}.$$
Define
\begin{equation}\label{vvv}
v(x):=C\left(\frac{1}{|x|^{\gamma}}-\frac{1}{|x|^{\nu-2s}}\right)
\end{equation}
where $\gamma=\frac{N-2s}{2}-\alpha_1$. Since $\gamma\ge \frac{N}{m}-2s$, using the fact that
$\nu<\frac{N}{m}$ it follows that $\gamma>\nu-2s$, and thus $v\ge 0$
in $B_1(0)$. Hence, choosing a suitable positive constant $C$, we
reach that
$$
(-\Delta)^s v -\lambda_1\dfrac{v}{|x|^{2s}}=f \inn B_1(0).
$$
Since $v\le 0$ in $\mathbb{R}^N\backslash B_1(0)$, by
comparison it follows that $v\le u$, where $u$ is the
unique solution of \eqref{mainfirst} for $\lambda=\lambda_1$. Since $\gamma\ge
\frac{N}{m}-2s$, then $v\notin L^{m^{**}_s}(B_1(0))$ and we
conclude.
\end{proof}
As a direct application of this lemma we can prove the optimality of the curve $J_s(m)$ in a particular case.
\begin{Lemma}\label{opt02}
Assume that the hypotheses of Lemma \ref{opt1} hold and let $u$ be the solution of problem \eqref{mainfirst}. If $m=\frac{2N}{N+2s}$ and $\lambda\ge J_{s}(m)$, then
$u\notin L^{m^{**}_s}(B_1(0))$.
\end{Lemma}
\begin{proof}
Notice that $P_s(m)$, defined in \eqref{Psm}, can be rewritten as
\begin{equation}\label{Psm2}
P_s(m)=\dfrac{2^{2s}\,\Gamma(\frac{N+2s}{2}-\frac{N}{2m})\Gamma(\frac{N}{2m})}{\Gamma(\frac{N}{2}-\frac{N}{2m})\Gamma(\frac{N}{2m}-s)},
\end{equation}
and in the particular case of $m=\frac{2N}{N+2s}$, it satisfies
$$J_s(m)=P_s(m). $$
Hence $\l\ge P_s(m)$ by hypothesis, and we conclude applying Lemma \ref{opt1}.
\end{proof}
\begin{remark}
Notice that in the local case, $s=1$,
$\alpha=\sqrt{\Lambda_{N,1}-\lambda}$ and $P_1(m)=J_{1}(m)$ for
all $m\in [\frac{2N}{N+2},\frac{N}{2})$, and thus Lemma \ref{opt1} holds in the whole range.
\end{remark}
Next we show that for radial functions the result in Lemma \ref{opt1} cannot be improved. In other words the optimality of the curve $J_s(m)$ cannot be proved with radial functions.We start  by proving the following result (for the properties of the Gamma function we refer to \cite{Abra}).
\begin{Lemma}\label{opt3}
Assume that $s\in (0,1)$ and $m\in [\frac{2N}{N+2s},\frac{N}{2s})$. Then
\begin{equation}\label{op0}
J_{s}(m)\le P_s(m)
\end{equation}
and equality holds in \eqref{op0} if and only if $m=\frac{2N}{N+2s}$.
\end{Lemma}
\begin{proof}
Using \eqref{Psm2} and the definition of $\Lambda_{N,s}$ (see \eqref{cteHardy}), it easily follows that
\eqref{op0} is equivalent to
\begin{equation}\label{TT10}
D(m):=
\dfrac{m^2}{(m-1)(N-2sm)}\dfrac{\Gamma(\frac{N+2s}{2}-\frac{N}{2m})\Gamma(\frac{N}{2m})}{\Gamma(\frac{N}{2}-\frac{N}{2m})\Gamma(\frac{N}{2m}-s)}\ge
\dfrac{4N}{(N-2s)^2}\dfrac{\Gamma^2(\frac{N+2s}{4})}{\Gamma^2(\frac{N-2s}{4})}=:\Theta(N,s).
\end{equation}
Furthermore, as we saw in the proof of Lemma \ref{opt02},
$$
D\left(\frac{2N}{N+2s}\right)=\Theta(N,s),
$$
and thus \eqref{TT10} holds if we prove that $D$ is an increasing function. Let us denote $D(m)=D_1(m)D_2(m)$ where
$$
D_1(m):=\dfrac{\Gamma(\frac{N+2s}{2}-\frac{N}{2m})\Gamma(\frac{N}{2m})}{\Gamma(\frac{N}{2}-\frac{N}{2m})\Gamma(\frac{N}{2m}-s)}\qquad\mbox{
and   } \qquad D_2(m):= \dfrac{m^2}{(m-1)(N-2sm)}.
$$
On the other hand, it is known that for $t>0$ there holds $\Gamma'(t)=\psi(t)\Gamma(t)$ where $\psi(t)$, the so called Digamma function, is given by
$$
\psi(t):=-\frac{1}{t}-C_0+t\sum_{n=1}^\infty\frac{1}{n(n+t)},
$$
with $C_0$ the Euler constant. Hence, it follows that
$$
D'_1(m)=\frac{N}{2m^2}D_1(m)K(m),
$$
where
$$
K(m)=\Big[\psi(a)-\psi(b)+\psi(c)-\psi(d)\Big],
$$
and
$$
a:=\frac{N+2s}{2}-\frac{N}{2m}, \qquad b:=\frac{N}{2}-\frac{N}{2m}, \qquad
c:=\frac{N}{2m}-s,\qquad d:=\frac{N}{2m}.
$$
Thus
$$
D'(m)=D_1(m)\frac{m}{(m-1)(N-2sm)}\left(\frac{(m-2)N+2sm}{(m-1)(N-2sm)}+\frac{N}{2m}K(m)\right).
$$
The first two terms here are positive, and then to analyze the sign of $D'$ we have to study the function
$$H(m):=\Big(\dfrac{(m-2)N+2sm}{(m-1)(N-2sm)}+\frac{N}{2m}K(m)\Big).$$
By definition there holds
$$
K(m)=-\dfrac{4sm^3}{N}\dfrac{(m-2)N+2sm}{(m-1)(N-2sm)((m-1)N+2sm)}+s\sum_{n=1}^\infty\left(
\frac{1}{(n+a)(n+b)}-\frac{1}{(n+c)(n+d)}\right),
$$
and noticing that $a=b+s$ and $d=c+s$, we have
$$
\frac{1}{(n+a)(n+b)}=\frac{1}{s}\Big(\frac{1}{n+b}-\frac{1}{n+b+s}\Big)
$$
and
$$
\frac{1}{(n+c)(n+d)}=\frac{1}{s}\Big(\frac{1}{n+c}-\frac{1}{n+c+s}\Big).
$$
Thus,
\begin{eqnarray*}
\dyle \frac{1}{(n+a)(n+b)}-\frac{1}{(n+c)(n+d)} &=
&\frac{1}{s}\left(\left(\frac{1}{n+b}-\frac{1}{n+c}\right)+\left(\frac{1}{n+c+s}-\frac{1}{n+b+s}\right)\right)\\
&=&-\frac{b-c}{s}\left(\frac{1}{(n+c)(n+b)}-\frac{1}{(n+c+s)(n+b+s)}\right).
\end{eqnarray*}
Hence $ H(m)=((m-2)N+2sm)H_1(m)$, where
$$
H_1(m):=
\frac{1}{(m-1)N+2sm}-\frac{N}{4m^2}\sum_{n=1}^\infty\Big[\frac{1}{(n+c)(n+b)}-\frac{1}{(n+c+s)(n+b+s)}\Big].
$$
Using the fact that $m\in (\frac{2N}{N+2s},\frac{N}{2s})$, we obtain
$(m-2)N+2sm>0$ and therefore the sign of $H$ is the sign of $H_1$. On the other hand, since $s\in (0,1)$, then
$$
\sum_{n=1}^\infty\Big[\frac{1}{(n+c)(n+b)}-\frac{1}{(n+c+s)(n+b+s)}\Big]\le
\frac{1}{(1+c)(1+b)},
$$
whence we conclude that
\begin{eqnarray*}
\dyle H_1(m) &\ge &
\frac{1}{(m-1)N+2sm}-\frac{N}{4m^2}\frac{1}{(1+c)(1+b)}\\
&=&\frac{1}{(m-1)N+2sm}-\frac{N}{(N+2m-2sm)((m-1)N+2m)}.
\end{eqnarray*}
Now, using the fact that $s\in (0,1)$, we get $H_1(m)>0$ and the
result follows.
\end{proof}

As a consequence we get the next regularity result.
\begin{Proposition}\label{propA0}
Let $s\in (0,1)$. Assume that $m\in (\frac{2N}{N+2s},\frac{N}{2s})$, $s<1$ and
$f(x)=\frac{1}{|x|^{\nu}}$ with $\nu=\frac{N-\e}{m}$ for some
small $\e>0$. Then there exists
$\tilde{\lambda}\in (J_s(m), P_s(m))$ such that the solution $u$
of problem \eqref{mainfirst} with $\lambda=\tilde{\lambda}$ satisfies $u\in L^{m_s^{**}}(B_1(0))$.
\end{Proposition}
\begin{proof}
Fix $\delta>0$ small enough so that if $\alpha\in (\alpha_0(m),
\alpha_0(m)+\delta)$, then $\gamma:=\frac{N-2s}{2}-\alpha$
satisfies $\gamma>\nu-2s$ and $\gamma m^{**}_s<N$.

Define $\tilde{\lambda}=\lambda(\alpha)$. Since $\alpha>\alpha_0$ then $\tilde{\lambda}<P_s(m)$. Thus, due to the continuity of $\lambda$ as a function of $\alpha$, choosing $\delta$ small enough and applying Lemma \ref{opt3} we deduce
$$J_s(m)<\tilde{\lambda}<P_s(m).$$
Let $u$ be the unique solution of problem \eqref{mainfirst} with $\lambda=\tilde{\lambda}$ and
consider the function $v$ defined in \eqref{vvv}.
Since $\gamma>\nu-2s$, by similar arguments as in the proof of Lemma \ref{opt1} we conclude that $v\le u$.
By setting $w:=u-v$, it follows that
\begin{equation*}
\left\{
\begin{array}{rcl}
(-\Delta)^s w-\lambda\dfrac{w}{|x|^{2s}} &=& 0 \inn\Omega,\\
w&=& -v\ge 0\inn \mathbb{R}^{N}\setminus\Omega.
\end{array}\right.
\end{equation*}
Thus, by Lemma \ref{Lm:singularity_elip} and Theorem \ref{singular}, $w\simeq |x|^{-\gamma}$ close to the origin. Since also $v\simeq |x|^{-\gamma}$, also  $u\simeq |x|^{-\gamma}$.
By  the definition of $\gamma$, we obtain that
$\gamma m^{**}_s<N$, thus $u\in L^{m^{**}_s}(B_1(0))$ and then we
conclude.
\end{proof}

\subsection{Nonvariational setting: weak solutions }

In this subsection we consider  $f\in L^ {m}(\Omega)$, with $1<m<\frac{2N}{N+2s}$.

\begin{Theorem}\label{nonv1}
Assume $1<m<\frac{2N}{N+2s}$ and $\lambda<J_s(m)$ defined in \eqref{Jsm}. Then problem \eqref{mainfirst} has a unique weak solution $u$ and it verifies
\begin{equation}\label{nonv2}
\|{u}\|_{L^{m^{**}_s}(\Omega)} \leq c\,\|{f}\|_{L^ {m}(\Omega)}
\qquad \mbox{ where } \qquad m^{**}_s=\frac{mN}{N-2ms}\, .
\end{equation}
Moreover, $u\in W^{s_1,m^*_s}_0(\Omega)$, for all $s_1<s$, with
$m^{*}_s=\frac{mN}{N-ms}$.
\end{Theorem}
\begin{proof}
Let $\{f_n\}_n\subset L^\infty(\Omega)$ be such that $0\le f_n\le
f$ and $f_n\uparrow f$ strongly in $L^m(\Omega)$. Define $u_n$ to
be the unique positive solution to the approximated problem
\begin{equation}\label{aprmain}
\left\{
\begin{array}{rcl}
(-\Delta)^s u_n -\lambda\dfrac{u_n}{|x|^{2s}+\frac 1n} &=& f_n \inn\Omega,\\
u_n&=&0\inn \mathbb{R}^{N}\setminus\Omega.
\end{array}\right.
\end{equation}
Then $\{u_n\}_n$ is monotone in $n$. As in the  proof of Theorem \ref{CZ} we use $u_n^\beta$ as a test function in
\eqref{aprmain} whit
$$0<\beta:=\dfrac{N(m-1)}{N-2ms}<1$$
(actually we have to test with $(u_n+\delta)^\beta$, $\delta>0$, and to make $\delta\rightarrow 0$ at the end, but to simplify we will drop this parameter here). Then,
\begin{eqnarray}\label{probT}
\frac{a_{N,s}}{2}\iint_{D_\Omega}\dfrac{(u_n(x)-u_n(y))(u_n^{\beta}(x)-u_n^{\beta}(y))}{|x-y|^{N+2s}}\,dx\,dy\le\lambda
\int_{\Omega}\dfrac{u_n^{\beta+1}}{|x|^{2s}}dx+ \int_{\Omega} f_n
u_n^{\beta} dx.
\end{eqnarray}
By the fact that $m'\beta=\frac{(\beta+1)2^*_s}{2}$ and by H\"older's inequality,
\begin{equation}\label{probT2}
\int_{\Omega} f_n u_n^{\beta} dx\le
\|f\|_{L^m(\Omega)}\Big(\int_{\Omega}u_n^{\frac{(\beta+1)2^*_s}{2}}
dx\Big)^{\frac{\beta}{\beta+1}\frac{2}{2^*_s}}.
\end{equation}
Now, using Lemma \ref{alg} and Hardy and Sobolev inequalities, it follows that
$$
\frac{a_{N,s}}{2}\Big(\frac{4\beta}{(\beta+1)^2}-\frac{\lambda}{\Lambda_{N,s}}\Big)
\Big(\iint_{D_\Omega}\dfrac{(u_n^{\frac{\beta+1}{2}}(x)-u_n^{\frac{\beta+1}{2}}(y))^2}{|x-y|^{N+2s}}
\,dx\,dy\Big)^{\frac{1}{\beta+1}}\le
C(N,s)\|f\|_{L^m(\Omega)}.$$
Since $\lambda<J_s(m)$, noticing that $m^{**}_s=2^*_s\frac{\beta+1}{2}$ and applying Sobolev inequality, \eqref{nonv2} follows.

Furthermore, in particular
\begin{equation}\label{probT3}
\int_{\Omega}u_n^{\frac{(\beta+1)2^*_s}{2}}
dx\le C_1\mbox{   and   }\int_{\Omega}\frac{u_n^{\beta+1}}{|x|^{2s}}dx
\le C_2,
\end{equation}
where $C_1$ and $C_2$ are independent of $n$.
Since $\frac{(\beta+1)2^*_s}{2}>1$, the  Lebesgue Theorem implies that  $u_n\uparrow u$ a.e. and strongly in $L^\sigma(\O)$ for all $1\le \sigma\le \frac{(\beta+1)2^*_s}{2}$ and $\dfrac{u_n}{|x|^{2s}}\uparrow\dfrac{u}{|x|^{2s}}\mbox{  strongly
in } L^1(\Omega)$. Therefore, $u$ is a weak solution of \eqref{mainfirst}.

Moreover by  using Fatou's Lemma, \eqref{probT2} and \eqref{probT3} in \eqref{probT} we obtain
\begin{equation}\label{TR}
\iint_{D_\Omega}\dfrac{(u(x)-u(y))(u^{\beta}(x)-u^{\beta}(y))}{|x-y|^{N+2s}}\,dx\,dy\le
C.
\end{equation}

Moreover $u$ is the unique weak solution. Indeed
if $u_2$ is an other solution of \eqref{mainfirst} with the above
regularities, then setting $w=u_2-u$, we conclude that
$$\left\{
\begin{array}{rcl}
(-\Delta)^s w-\lambda\dfrac{w}{|x|^{2s}} &=& 0\inn\Omega,\\
w &=&0\inn \mathbb{R}^{N}\setminus\Omega.
\end{array}\right.
$$
By testing with  $\phi\in \mathcal{T}$ defined in \eqref{test}, with $(-\Delta)^s \phi=\varphi>0$, we obtain that $u_2\equiv u$.

To finish we prove the regularity of the {\it fractional gradient}. Fix $s_1<s$ and let $q=m^*_s<2$.
Call
$$
d\sigma:=\begin{cases}
\Big(\dfrac{u^\beta(x)-u^\beta(y)}{u(x)-u(y)}\Big)dxdy \qquad\hbox{ if }u(x)\neq u(y),\\
0 \qquad\hbox{ if }u(x)=u(y),
\end{cases}$$
and notice that $d\sigma$ is positive. Therefore, by Hölder's inequality
\begin{equation*}\begin{split}
&\dyle\int_{\Omega}\int_{\Omega}\dfrac{|u(x)-u(y)|^q}{|x-y|^{N+qs_1}}\,dx\,dy=
\int_{\Omega}\int_{\Omega}\dfrac{|u(x)-u(y)|^q}{|x-y|^{N+qs_1}}\Big(\dfrac{u^\beta(x)-u^\beta(y)}{u(x)-u(y)}\Big)\times
\Big(\dfrac{u(x)-u(y)}{u^\beta(x)-u^\beta(y)}\Big)\,dx\,dy\\
&\le
\Big(\dyle\int_{\Omega}\int_{\Omega}\dfrac{|u(x)-u(y)|^2}{|x-y|^{N+2s}}d\sigma\Big)^{\frac{q}{2}}
\times \dyle \Big(\int_{\Omega}\int_{\Omega}\Big(\dfrac{u(x)-u(y)}{u^\beta(x)-u^\beta(y)}\Big)^{\frac{2}{2-q}}\dfrac{d\sigma}{|x-y|^{N-\theta}}\Big)^{\frac{2-q}{2}}\\
&=\Big(\dyle\int_{\Omega}\int_{\Omega}\dfrac{(u(x)-u(y))(u^\beta(x)-u^\beta(y))}{|x-y|^{N+2s}}dxdy\Big)^{\frac{q}{2}}
\times
\Big(\int_{\Omega}\int_{\Omega}\Big(\dfrac{u(x)-u(y)}{u^\beta(x)-u^\beta(y)}\Big)^{\frac{q}{2-q}}\dfrac{dxdy}{|x-y|^{N-\theta}}\Big)^{\frac{2-q}{2}},
\end{split}\end{equation*}
where $\theta:=\frac{2(s-s_1)}{2-q}$.  The first term is bounded by \eqref{TR}, and the second one can be estimated as follows.
Since $\beta<1$, then
$$
0\le \dfrac{u(x)-u(y)}{u^\beta(x)-u^\beta(y)}\le
\frac{1}{\beta}(u^{1-\beta}(x)+u^{1-\beta}(y)),
$$
and hence
\begin{eqnarray*}
\dyle
\int_{\Omega}\int_{\Omega}\Big(\dfrac{u(x)-u(y)}{u^\beta(x)-u^\beta(y)}\Big)^{\frac{q}{2-q}}\dfrac{dxdy}{|x-y|^{N-\theta}}
&\le & \frac{C}{\beta}\int_{\Omega}\int_{\Omega}((u^{\frac{(1-\beta)q}{2-q}}(x)+u^{\frac{(1-\beta)q}{2-q}}(y))\dfrac{dxdy}{|x-y|^{N-\theta}}\\
&\le &
\frac{2C}{\beta}\int_{\Omega}u^{\frac{(1-\beta)q}{2-q}}(x)\left(\int_{\Omega}\dfrac{dy}{|x-y|^{N-\theta}}\right)dx.
\end{eqnarray*}
Notice that
$$
\sup_{\{x\in
\Omega\}}\left(\int_{\Omega}\dfrac{dy}{|x-y|^{N-\theta}}\right)\le C,
$$
hence
$$
\dyle
\int_{\Omega}\int_{\Omega}\Big(\dfrac{u(x)-u(y)}{u^\beta(x)-u^\beta(y)}\Big)^{\frac{q}{2-q}}\dfrac{dxdy}{|x-y|^{N-\theta}}\le
C_1\int_{\Omega}u^{\frac{(1-\beta)q}{2-q}}(x)dx.
$$
Now, since $\dfrac{(1-\beta)q}{2-q}=m^{**}_s$, the result
follows using \eqref{nonv2}.
\end{proof}

In the case where no condition is imposed on $\lambda$, then
additional condition on $f$ is needed. The next Theorem gives a necessary and sufficient
condition to ensure the existence of a weak solution.

\begin{Theorem}\label{necesariosuficiente}
Let $\l\le    \Lambda_{N,s}$ and suppose that $f\in L^1(\Omega)$, $f\ge 0$.
Then $u$ is a positive weak solution to the problem
\eqref{mainfirst} if and only if $f$ satisfies
$$\dint_{B_{r}(0)}|x|^{-\gamma}f\,dx<+\iy,$$
for some $B_{r}(0)\subset\subset\Omega$.

Moreover, if $u$ is the unique weak solution to \eqref{mainfirst},
then
\begin{equation} \label{tku}
\forall k \geq 0 \qquad T_k (u)  \in H^s_0 (\Omega), \qquad\, u \in L^
{q}(\Omega) \,, \qquad  \forall  \ q\in \big(1,
\frac{N}{N-2s}\big)\,,
\end{equation}
 \begin{equation}\label{L1du}
\big|(-\Delta)^{\frac{s}{2}}   (u)\big| \in L^{r}(\Omega) \,,
\qquad \forall  \  r \in \big(1,  \frac{N}{N-s} \big) \,,
\end{equation}
and $u\in W^{s_1,q_1}_0(\Omega)$ for all $s_1<s$ and for all
$q_1<\frac{N}{N-s}$.
\end{Theorem}
\proof \textit{Necessary condition: } Consider
$u$  a positive weak solution to the problem \eqref{mainfirst} and
consider $\varphi_n\in L^{\infty}(\Omega)\cap H_{0}^{s}(\Omega)$ the
positive solution to
$$
\left\{\begin{array}{rcl}
(-\Delta)^s\varphi_n&=&\lambda\dfrac{\varphi_{n-1}}{|x|^{2s}+\frac{1}{n}}+1 \quad\mbox{in }\Omega,\\
\varphi_n&=&0\inn\ren\setminus\Omega,\\
\end{array}\right.
$$
where
$$
\left\{\begin{array}{rcl}
(-\Delta)^s \varphi_0&=&1 \quad\mbox{in }\Omega,\\
\varphi_0&=&0\inn\ren\setminus\Omega.\\
\end{array}\right.
$$
Then it is easy to check that    $\varphi_0\leq \varphi_{1}\leq
\varphi_{n-1}\leq\varphi_{n}\leq\varphi$, where $\varphi$ is the pointwise limit and then
\begin{equation}\label{varphiProb}
\left\{\begin{array}{rcl}
(-\Delta)^s \varphi&=&\lambda\dfrac{\varphi}{|x|^{2s}}+1 \quad\mbox{in }\Omega,\\
\varphi&>&0\quad\hbox{ in }\Omega,\\
\varphi&=&0\inn(\ren\setminus\Omega).\\
\end{array}\right.
\end{equation}
Taking $\varphi_n$ as a test function in \eqref{mainfirst}, we get
$$\dint_{\Omega} f\varphi_{n}\,dx\leq\dint_{\Omega}
u\,dx=C<\infty.$$
Hence, $\{f\varphi_{n}\}_{n\in\mathbb{N}}$ is an increasing
sequence uniformly bounded in $L^1(\Omega)$, and then applying the
Monotone Convergence Theorem and Lemma \ref{Lm:singularity_elip} we obtain
$$
\tilde{C} \dint_{B_{r}(0)} |x|^{-\gamma} f\,dx\leq \dint_{\Omega}
f\varphi\,dx \leq C.
$$

\noindent \textit{Sufficient condition:} Assume that
$$\dint_{B_{r}(0)}|x|^{-\gamma}f\,dx<+\iy,$$ for all
$B_{r}(0)\subset\subset\Omega$ small enough; let consider the
sequence of energy solutions $u_n\in L^{\infty}(\Omega)\cap
H_{0}^{s}(\Omega)$ to the following approximated problems
\begin{equation}\label{approx}
\left\{
\begin{array}{rcl}
(-\Delta)^s u_{n}&=&\lambda\dfrac{u_{n-1}}{|x|^{2s}+\frac{1}{n}}+f_n\inn \Omega,\\
u_{n}(x)&=&0\inn \ren\setminus\Omega,
\end{array}
\right.
\end{equation}
where
$$
\left\{
\begin{array}{rcl}
(-\Delta)^{s} u_{0}&=&f_1\inn \Omega,\\
u_{0}(x)&=&0\inn (\ren\setminus\Omega),\\
\end{array}
\right.
$$
with $f_n= T_{n}(f)$ and $u_{0}\leq u_{1}\leq u_{n-1}\leq
u_{n}\inn\mathbb{R}^{N}$. Since $f_n\ge 0$, $u_n(x)\ge 0$ in $\Omega$.   Take  $\varphi\in
H^{s}_{0}(\Omega)$, the positive energy solution to \eqref{varphiProb},
as a test function in \eqref{approx}. As a consequence of Lemma \ref{singular}, it follows that
$$\dint_{\Omega}
u_n\,dx\leq \dint_{\Omega} f\varphi\,dx\leq \tilde{C}\int_\Omega f|x|^{-\gamma}\leq C.$$
Hence, since the
sequence $\{u_n\}_{n\in\mathbb{N}}$ is increasing, we can define
$u:=\lim_{n\rightarrow\infty}u_n$, and conclude that $u\in
L^1(\Omega)$. We
claim that $\dfrac{u}{|x|^{2s}}\in L^1(\Omega)$.
Indeed, let  $\psi$ be the unique bounded positive solution to the
problem $$ \left\{\begin{array}{rcl}
(-\Delta)^s \psi&=&1 \quad\mbox{in }\Omega,\\
\psi &=&0\inn \ren\setminus\Omega, \\
\end{array}\right.
$$
then  $\psi\ge C$ in $B_r(0)$.  By using $\psi$ as a
test function in \eqref{approx},$$
\int_{\Omega}\frac{u_{n-1}}{|x|^{2s}+\frac{1}{n}}dx\le \frac 1C\int_{B_r(0)}\frac{\psi u_{n-1}}{|x|^{2s}+\frac{1}{n}}dx+ C(r)\int_{\Omega\backslash B_r(0)}u_ndx\le C,
$$
and thus
$$\lambda\dfrac{u_{n-1}}{|x|^{2s}+\frac{1}{n}}+f_n \nearrow \lambda\dfrac{u}{|x|^{2s}}+f \mbox{ strongly in }L^1(\Omega).$$

Testing with $T_k(u_n)$ in \eqref{approx} and considering the previous estimates, we easily get that
$T_k(u_n)\rightharpoonup T_k(u)$ weakly in $H^s_0 (\Omega)$.
Since the sequence $\{\lambda\dfrac{u_{n-1}}{|x|^{2s}+\frac{1}{n}}+f_n\}_{n\in\mathbb{N}}$ converges strongly in $L^1(\Omega)$, then by the results of \cite{LPPS}, we reach that $u\in L^\sigma(\Omega)$ for all
$\sigma<\frac{N}{N-2s}$ and $\big|(-\Delta)^{\frac{s}{2}}(u)\big|
\in L^{r}(\Omega)$ for all  $r \in \big(1,
\frac{N}{N-s} \big)$. Moreover, according to Theorem 5 (C) of Charter 5 in \cite{Stein}, we conclude that $u\in W^{s_1,q_1}_0(\Omega)$ for all $s_1<s$ and for all
$q_1<\frac{N}{N-s}$. See too \cite{KMS} and  \cite{AAB} for a simple proof.
\qed
\medskip

\begin{remark}\label{nonexistrem}
As a consequence of this result, together with the weak Harnack inequality, one can easily prove nonexistence for $\lambda>\Lambda_{N,s}$.
\end{remark}

\section{Problems with the Hardy potential and nonlinear term  singular at the boundary.}\label{sec:main}\rm
The results in this section have some partial precedents in \cite{AA} and are also applicable to the    local case.
The aim will be to study the problem
\begin{equation}\label{mainp}
\left\{
\begin{array}{rcll}
(-\Delta)^{s}u &= &\lambda \dfrac{u}{|x|^{2s}}+\dfrac{h(x)}{u^{\sigma}}&{\rm in}\; \Omega,\\
u &> & 0 &{\rm in}\; \Omega, \\
u &= &0 &{\rm in } \,\, \ren\setminus\Omega,
\end{array}
\right.
\end{equation}
where $h$ is a nonnegative function, $\sigma>0$
and $\lambda\ge 0 $. As we pointed out in Remark \ref{nonexistrem} it can be easily checked that problem \eqref{mainp} has
no positive solution for $\l>\Lambda_{N,s}$. Hence we will assume
$\l\le \Lambda_{N,s}$.
\begin{remark} Call $\mu:=-\sigma$.
We know that problem \eqref{mainp} has no positive solution for $\mu>p_+(\lambda):= 1+\frac{2s}{\gamma}$, where $\gamma$ is defined in \eqref{g1}.  A quite complete study is done in \cite{BMP}, also for the case $1< \mu<p_+(\lambda)$ (see \cite{F} for a different approach). The case $\mu=1$ is related to the first eigenvalue of the operator $ (-\Delta)^{s}(\cdot)-\lambda \dfrac{(\cdot)}{|x|^{2s}}$ and the case $0\le \mu<1$ is easily handled as a minimization problem.

Therefore, finding a solution of \eqref{mainp} can be seen as proving that there is not a lower threshold for the power to solve the semilinear problem.
\end{remark}

The main existence result in this section is the following.
\begin{Theorem}\label{m1}
Assume that $\sigma \ge 1$ and $\lambda\le
\Lambda_{N,s}$. Then, for all $h\in L^1(\Omega)$, problem
\eqref{mainp} has a positive weak solution. More
precisely,
\begin{enumerate}
\item if $\sigma=1$, then $u\in H^{s}_0(\Omega)$
whether $\l<\Lambda_{N,s}$, and $u\in W^{s,q}_0(\Omega)$ for all $q<2$
if $\l=\Lambda_{N,s}$;
\item if $\sigma>1$, then
$u\in H^{s}_{loc}(\Omega)$ with $G_k(u)\in H^{s}_0(\Omega)$ and
$T_k^{\frac{\sigma+1}{2}}(u)\in H^s_0(\Omega)$.
Moreover if
$\Big[\dfrac{4\sigma}{(\sigma+1)^2}-\dfrac{\lambda}{\Lambda_{N,s}}\Big]>0$,
then $u^{\frac{\sigma+1}{2}}\in H^{s}_0(\Omega)$.
\end{enumerate}

\end{Theorem}
\begin{proof} Let $h_n:=T_n(h)$, the usual truncation of $h$, and define $u_n$ to be the unique
positive solution to the approximated problem
\begin{equation}\label{apro}
\left\{
\begin{array}{rcll}
(-\Delta)^{s}u_n &= &\lambda \dfrac{u_n}{|x|^{2s}}+\dfrac{h_n(x)}{(u_n+\frac{1}{n})^{\sigma}} &{\rm in}\; \Omega,\\
u_n &> & 0 &{\rm in}\; \Omega, \\
u_n &= & 0&{\rm in}\; (\ren\setminus\Omega).
\end{array}
\right.
\end{equation}
The existence follows by minimization and the uniqueness by using the result in Lemma \ref{compar}.
Since $T_n(h)$ is an increasing function in $n$, again by Lemma \ref{compar} we conclude that
$\{u_n\}_{n\in\mathbb{N}}$ is an increasing function in $n$. We divide the proof in two cases.
\medskip

{\bf First case: $\sigma=1$ and
$\l<\Lambda_{N,s}$.}

Taking $u_n$ as a test function in \eqref{apro} and using the Hardy
inequality, we obtain
\begin{equation*}\label{ess}
\frac{a_{N,s}}{2}\left(1-\frac{\l}{\Lambda_{N,s}}\right)\|u_n\|_{H^{s}_0(\Omega)}^2\leq\int_{\Omega}\dfrac{h_n
u_n}{u_n+\frac{1}{n}}\,dx\leq \int_{\Omega}h \,dx=C.
\end{equation*}

Thus $\{u_n\}_{n\in\mathbb{N}}$ is bounded in $H^{s}_0(\Omega)$ and then there exists
$u\in H^{s}_0(\Omega)$ such that, up to a subsequence, $u_n\rightharpoonup u$ weakly in
$H^{s}_0(\Omega)$ and $u_n\uparrow u$ strongly in $L^\eta(\Omega)$
for all $\eta< 2^*_s$.

Since $(-\Delta)^{s}u_n\ge 0$, using
the monotonicity of $\{u_n\}_{n\in\mathbb{N}}$ and the compactness Lemma
\ref{compa} we easily obtain that $u_n\to u$ strongly in
$H^{s}_0(\Omega)$. Hence we conclude that $u$ solves problem
\eqref{mainp}.
\medskip

{\bf Second case: $\sigma>1$.}

Using now $G_k(u_n)$ as a test function in
\eqref{apro} we have
\begin{equation*}\label{estim1}
\frac{a_{N,s}}{2}\iint_{D_\Omega}{\frac{|G_k(u_n(x))-G_k(u_n(y))|^2}{|x-y|^{N+2s}}}\,
dx\, dy
-\lambda\int_{\Omega}\dfrac{u_nG_k(u_n)}{|x|^{2s}}\,dx\leq\int_{\Omega}\dfrac{h_n
G_k(u_n)}{(u_n+\frac{1}{n})^\sigma}\,dx
\end{equation*}
and
$$\int_{\Omega}\dfrac{h_n G_k(u_n)}{(u_n+\frac{1}{n})^\sigma}\,dx\le
\frac{1}{k^{\sigma-1}}\int_{\Omega} h \,dx.
$$
Moreover, $u_nG_k(u_n)=G^2_k(u_n)+kG_k(u_n)$, and thus
\begin{equation*}\label{estim20}
\begin{split}
\frac{a_{N,s}}{2}&\iint_{D_\Omega}\frac{|G_k(u_n(x))-G_k(u_n(y))|^2}{|x-y|^{N+2s}}\,
dx\, dy-\lambda\int_{\Omega}\dfrac{G_k^2(u_n)}{|x|^{2s}}\,dx\\
& \qquad \;\;\;\leq\lambda k\int_{\Omega} \dfrac{G_k(u_n)}{|x|^{2s}}\,dx
+\frac{1}{k^{\sigma-1}}\int_{\Omega} h \,dx.
\end{split}
\end{equation*}
Taking into account that $\lambda<\Lambda_{N,s}$, by the Hardy-Sobolev inequality we
obtain
\begin{equation*}\label{estim4}
C\iint_{D_\Omega}{\frac{|G_k(u_n(x))-G_k(u_n(y))|^2}{|x-y|^{N+2s}}}\, dx\,
dy\leq \lambda k\int_{\Omega} \dfrac{G_k(u_n)}{|x|^{2s}}\,dx
+C(k,h),
\end{equation*}
and applying Young and Hardy-Sobolev inequalities on the integral in the right hand side
we reach that
\begin{equation*}\label{estim5}
\iint_{D_\Omega}{\frac{|G_k(u_n(x))-G_k(u_n(y))|^2}{|x-y|^{N+2s}}}\, dx\,
dy\leq C(k,\lambda,\Lambda_{N,s},h).
\end{equation*}
Therefore $\{G_k(u_n)\}_{n\in\mathbb{N}}$ is uniformly bounded in $H^s_0(\Omega)$, and again by the Hardy-Sobolev
inequality,
$$
\int_{\Omega}\dfrac{G^2_k(u_n(x))}{|x|^{2s}}dx\le
C(k,\lambda,\Lambda_{N,s},h).
$$
Then we get
$$
\int_{\Omega}\dfrac{u^2_n(x)}{|x|^{2s}}dx=\int_{\Omega}\dfrac{T_k^2(u_n(x))}{|x|^{2s}}dx+\int_{\Omega}
\dfrac{G^2_k(u_n(x))}{|x|^{2s}}dx
+2\int_\Omega\frac{T_k(u_n)G_k(u_n)}{|x|^{2s}}\,dx\le
C(k,\lambda,\Lambda_{N,s},h).
$$

\noindent Likewise, testing with $T_k^\sigma(u_n)$ in
\eqref{apro}, it follows that
\begin{equation*}\label{estim6}
\begin{split}
\dyle
\frac{a_{N,s}}{2}&\iint_{D_\Omega}{\frac{(T_k^\sigma(u_n(x))-T_k^\sigma(u_n(y)))(u_n(x)-u_n(y))}{|x-y|^{N+2s}}}
dx\, dy\\
&\leq\dyle
\lambda\int_{\Omega}\dfrac{u_nT^{\sigma}_k(u_n)}{|x|^{2s}}\,dx+\int_{\Omega}\dfrac{h_n
T^\sigma_k(u_n)}{(u_n+\frac{1}{n})^\sigma}\,dx\\
&\le \dyle
k^{\sigma-1}\lambda\int_{\Omega}\dfrac{u^2_n}{|x|^{2s}}\,dx+\int_{\Omega}h_n
\,dx\le C(k,\lambda,\Lambda_{N,s},h),
\end{split}
\end{equation*}
and applying Lemma \ref{alg} we conclude
$$
\iint_{D_\Omega}{\frac{(T^{\frac{\sigma+1}{2}}_k(u_n(x))-T^{\frac{\sigma+1}{2}}_k(u_n(y)))^2}{|x-y|^{N+2s}}}\,
dx\, dy\leq C(k,\lambda,\Lambda_{N,s},h,\sigma).
$$
Thus $\{T^{\frac{\sigma+1}{2}}_k(u_n)\}_{n\in\mathbb{N}}$ is bounded in
$H^{s}_0(\Omega)$. Furthermore, the strong maximum principle provides that
$$u_n\ge u_1\ge c(K)>0, \hbox{ for any compact set  } K\subset \O.$$

\textit{Claim}.-  $\{T_k(u_n)\}_{n\in\mathbb{N}}$ is bounded in
$H^{s}_{loc}(\Omega)$.

Since $\{u_n\}_{n\in\mathbb{N}}$ is an increasing
sequence, then $T_k(u_n)\ge T_k(u_1)$,  and for all $\Omega'\subset\subset \Omega$, $u_1\ge C(\Omega')$. Thus,
$$T_k(u_n)\ge \min\{k,C(\Omega')\}=: {C_0}.$$

\noindent For $(x,y)\in \Omega'\times\Omega'$, we define
$\alpha_n:=\dfrac{T_k(u_n(x))}{{C_0}}$ and
$\beta_n:=\dfrac{T_k(u_n(y))}{{C_0}}$. It is clear that
$\alpha_n, \beta_n\ge 1$.
Therefore the following inequality holds,
\begin{equation}\label{alphabeta}
(\alpha_n-\beta_n)^2\le (\alpha^{\frac{\sigma+1}{2}}_n
-\beta^{\frac{\sigma+1}{2}}_n)^2.
\end{equation}
Indeed, if $\alpha_n=\beta_n$ the estimate is trivial. Otherwise, without loss of generality  we can assume $\alpha_n>\beta_n\ge 1$.
Let $0<x:=\frac{\beta_n}{\alpha_n}<1$. Since $\sigma>1$, we easily obtain that
$$
0\le 1-x\le 1-x^{\frac{\sigma+1}{2}},
$$
and hence
$$
(1-x)^2\le (1-x^{\frac{\sigma+1}{2}})^2.
$$
Clearly $\alpha^2_n<\alpha^{\sigma+1}_n$, and thus
$$
\alpha^2_n(1-x)^2\le
\alpha^{{\sigma+1}}_n(1-x^{\frac{\sigma+1}{2}})^2.
$$
Recalling the definition of $x$, \eqref{alphabeta} follows.

Finally, by the definition of $\alpha_n$ and $\beta_n$, we
conclude that for $(x,y)\in \Omega'\times\Omega'$, we have
$$
(T_k(u_n(x))-T_k(u_n(y)))^2\le
{C_0}^{1-\sigma}(T^{\frac{\sigma+1}{2}}_k(u_n(x))-T^{\frac{\sigma+1}{2}}_k(u_n(y)))^2.
$$
Thus the claim follows using the boundedness of $\{T^{\frac{\sigma+1}{2}}_k(u_n)\}_{n\in\mathbb{N}}$
in $H^{s}_0(\Omega)$.

Hence combining the estimates above, we obtain that $\{u_n\}_{n\in\mathbb{N}}$ is
bounded in $H^{s}_{loc}(\Omega)$ and then, up to a subsequence, there exists
$u\in H^{s}_{loc}(\Omega)$ such that
$$\begin{array}{rcl}
u_n&\rightharpoonup& u \hbox{   weakly in } H^{s}_{loc}(\Omega),\\
G_k(u_n)&\rightharpoonup& G_k(u) \hbox{  weakly  in }   H^{s}_{0}(\Omega) \hbox{    and  }\\
T^{\frac{{\sigma}+1}{2}}_k(u) &\rightharpoonup&T^{\frac{{\sigma}+1}{2}}_k(u)  \hbox{  weakly  in }   H^{s}_{0}(\Omega).
\end{array}
$$
Applying  the compactness result in Lemma \ref{compag} we obtain
that $u_n\to u$ strongly in $H^{s}_{loc}(\Omega)$.
Thus $u_n\uparrow
u$ strongly in $L^r(\Omega)$ for all $1\le r< 2^*_s$.

Let $\phi\in \mathcal{T}$, where $\mathcal{T}$ was defined in \eqref{test}. Testing with $\phi$ in \eqref{apro}, it follows that
\begin{equation}\label{lastt}
\int_\Omega(-\Delta)^{s}u_n \phi dx=\lambda \int_\Omega
\dfrac{u_n\phi }{|x|^{2s}}dx+\int_\Omega\dfrac{\phi
 h_n(x)}{(u_n+\frac{1}{n})^\sigma}dx.
\end{equation}
By the estimates above we reach that, when $n$ goes to $+\infty$,
$$
\lambda \int_\Omega \dfrac{u_n\phi
}{|x|^{2s}}+\int_\Omega\dfrac{\phi
h_n(x)}{(u_n+\frac{1}{n})^{\sigma}}dx\to \lambda \int_\Omega \dfrac{u\phi
}{|x|^{2s}}+\int_\Omega\dfrac{\phi h(x)}{u^{\sigma}}dx<+\infty.
$$
Moreover, we have
$$
\int_\Omega(-\Delta)^{s}u_n \phi dx=\int_\Omega
u_n(-\Delta)^{s}\phi dx\to \int_\Omega u(-\Delta)^{s}\phi dx,
$$
as $n\rightarrow +\infty$. Therefore, passing to the limit in \eqref{lastt},
$$
\dint_{\Omega}(-\Delta)^{s}u\phi\,dx=\lambda \int_\Omega
\dfrac{u\phi}{|x|^{2s}}\,dx+\dint_{\Omega}\dfrac{h(x)\phi}{u^\sigma}\,dx,
$$
i.e., $u$ is a weak solution.

Finally, for every $\lambda<\Lambda_{N,s}$ take $\sigma$ such
that
$\Big[\dfrac{4\sigma}{(\sigma+1)^2}-\dfrac{\lambda}{\Lambda_{N,s}}\Big]>0$.
By using $u_n^{\sigma}$ as a test function in \eqref{apro}, it follows
that
\begin{eqnarray*}
\iint_{D_\Omega}\dfrac{(u_n(x)-u_n(y))(u^{\sigma}_n(x)-u^{\sigma}_n(y))}{|x-y|^{N+2s}}\,dx\,dy\le
\lambda \int_{\Omega}\dfrac{u^{\sigma+1}_n}{|x|^{2s}}dx+
\int_{\Omega} h_n dx.
\end{eqnarray*}
By  Lemma  \eqref{alg}, we get
$$
(u_n(x)-u_n(y))(u^{\sigma}_n(x)-u^{\sigma}_n(y))\ge
\frac{4\sigma}{(\sigma+1)^2}(u^{\frac{\sigma+1}{2}}_n(x)-u^{\frac{\sigma+1}{2}}_n(y))^2,
$$
and hence, by the Hardy inequality,
$$
\frac{a_{N,s}}{2}\left(\frac{4\sigma}{(\sigma+1)^2}-\frac{\lambda}{\Lambda_{N,s}}\right)
\iint_{D_\Omega}\dfrac{(u^{\frac{\sigma+1}{2}}_n(x)-u^{\frac{\sigma+1}{2}}_n(y))^2}{|x-y|^{N+2s}}
\,dx\,dy\le \|h\|_{L^1(\Omega)}.$$
Therefore $u^{\frac{\sigma+1}{2}}\in H^{s}_{0}(\Omega)$   and this is the sense how the boundary value is reached.
\end{proof}

We now deal with the case $\sigma<1$. If $h\in L^{(\frac{2^*_s}{1-\sigma s})'}(\Omega)$, the existence of a positive energy solution can be proved proceeding as in the case $\sigma=1$. However, our goal from now on will be to study the solvability when $h$ has less regularity. Indeed, we have the following result.
\begin{Theorem}\label{phop}
Assume $\sigma<1$, $\lambda<\Lambda_{N,s}$ and $h\in
L^1(\Omega,|x|^{-(1-\sigma)\gamma} dx)$, $h\gneq 0$. Then  problem
\eqref{mainp} has at least a weak solution.
\end{Theorem}
\begin{proof} Let $\{h_n\}_{n\in\mathbb{N}}$ be such that $h_n\ge 0$ and $h_n \uparrow h$ strongly in $L^1(\Omega)$. Define $u_n$ as the unique positive solution to
the approximated problem \eqref{apro}. Then by setting $v_n:=|x|^{\gamma}u_n$, it follows that $v_n$
satisfies
\begin{equation}\label{peso}
L_\gamma(v_n)=|x|^{-\gamma}\frac{h_n(x)}{\left(|x|^{-\gamma}v_n+\frac{1}{n}\right)^\sigma}
\leq|x|^{-(1-\sigma)\gamma}\frac{h_n}{v^\sigma_n}\qquad\hbox{ in }\Omega,
\end{equation}
where $L_\gamma$ was defined in \eqref{Ltilde0}. Using $v_n^\sigma$ as a test function in \eqref{peso} we obtain that
$$
\frac{a_{N,s}}{2}\frac{4\sigma}{(\sigma+1)^{2}}\iint_{D_\Omega}{\frac{(v_n^{\frac{\sigma+1}{2}}(x)-v_n^{\frac{\sigma+1}{2}}(y))^2}{|x-y|^{N+2s}}}
\dfrac{dx}{|x|^\g}\dfrac{dy}{|y|^\g}
\le\int_{\Omega}\dfrac{h_n}{|x|^{(1-\sigma)\gamma}} \,dx\leq C,
$$
with $C$ independent of $n$.

Thus we conclude that the sequence
$\{v_n^{\frac{\sigma+1}{2}}\}_{n\in\mathbb{N}}$ is bounded in the weighted
Sobolev space $ Y^{s,\g}_0(\O)$ and hence there exists $u_0$ such that, up to a subsequence,
$$v_n^{\frac{\sigma+1}{2}}\rightharpoonup v_0^{\frac{\sigma+1}{2}}\;\;\hbox{
weakly in } Y^{s,\g}_0(\O).$$ Since $L_{\g}(v_n)\ge 0$, then using
the fact that $\frac{\sigma+1}{2}<1$, also
$L_{\g}(v^{\frac{\sigma+1}{2}}_n)\ge 0$. Hence by the monotonicity
of $v_n$ and Lemma \ref{compa}, we obtain that
$$v_n^{\frac{\sigma+1}{2}}\to v_0^{\frac{\sigma+1}{2}}\;\;\hbox{
strongly in }Y^{s,\g}_0(\O).$$

\noindent Passing to the limit in \eqref{peso}, it follows that $v_0$
solves
\begin{equation*}
\left\{\begin{array}{rcll}
L_\g(v_0)&=&|x|^{-(1-\sigma)\gamma}\frac{h}{v_0^\sigma} &\hbox{  in } \Omega,\\
v_0&=&0 &\hbox{  in } \mathbb{R}^N\setminus\Omega,
\end{array}
\right.
\end{equation*}
in the weak sense.
Defining $u_0:=|x|^{-\gamma}v_0$, then $\lambda\dfrac{u_0}{|x|^{2s}}\in
L^1(\Omega)$ and $u_0$ is a weak solution of problem \eqref{mainp}.
\end{proof}

To end this section, we consider the problem in the whole space, that is, $\Omega=\ren$. Then we will work in
the space $\dot{H}^{s}(\ren)$ defined the completion of
$\mathcal{C}^\infty_0(\mathbb{R}^N)$ with respect to the Gagliardo seminorm
$$[\phi]_{\dot{H}^{s}(\ren)}=\Big(\int_{\mathbb{R}^N}\int_{\mathbb{R}^N}\dfrac{(\phi(x)-\phi(y))^2}{|x-y|^{N+2s}} \,dxdy\Big)^{\frac 12}.
$$
We obtain the following existence result.
\begin{Theorem}\label{rnn}
Consider the problem
\begin{equation}\label{rn}
\left\{
\begin{array}{rcll}
(-\Delta)^s u& = &\lambda \dfrac{u}{|x|^{2s}}+\dfrac{h(x)}{u^\sigma} &{\rm in}\; \mathbb{R}^N,\\
u &> & 0 &{\rm in}\; \mathbb{R}^N.
\end{array}
\right.
\end{equation}
Then
\begin{enumerate}
\item[(i)] If $\sigma=1$, then for all $h\in L^1(\mathbb{R}^N)$,
problem \eqref{rn} has a solution $u\in \dot{H}^{s}(\ren)$ .

\item[(ii)] If
$\sigma>1$, then for all $h\in L^1(\mathbb{R}^N)$, problem
\eqref{rn} has a weak solution $u$ such that $G_k(u)\in
\dot{H}^{s}(\ren)$ and $T_k^{\frac{\sigma+1}{2}}(u)\in \dot{H}^{s}(\ren)$,
for all $k>0$. Moreover, if
\begin{equation}\label{lambsig}
{{4\sigma}\over{(\sigma+1)^2}}-{{\lambda}\over
{\Lambda_{N,s}}}>0,
\end{equation}
then $u^{\frac{\sigma+1}{2}}\in
\dot{H}^{s}(\ren)$.

\item[(iii)] If $\sigma<1$ and $h\in L^m(\mathbb{R}^N)$ with
$m=(\frac{2^*_s}{1-\sigma})'$, then problem \eqref{rn} has a
solution $u$ such that $u\in\dot{H}^{s}(\ren)$.

\item[(iv)] If $\sigma<1$ and \eqref{lambsig} holds, then for all $h\in L^1(\mathbb{R}^N)$
problem \eqref{rn} has a weak solution $u$ such that
$u^{\frac{\sigma+1}{2}}\in \dot{H}^{s}(\ren)$.
\end{enumerate}
\end{Theorem}
\begin{proof}
Consider $u_n$ to be the unique positive solution to the
approximated problem
\begin{equation}\label{aprn}
\left\{
\begin{array}{rcll}
(-\Delta)^su_n&=&\lambda \dfrac{u_n}{|x|^{2s}+\frac{1}{n}}+\dfrac{h_n(x)}{(u_n+\frac{1}{n})^\sigma} &{\rm in}\; B_n(0),\\
u_n&>&0 &{\rm in}\;  B_n(0), \\
u_n&=&0&{\rm in}\; \ren \setminus B_n(0).
\end{array}
\right.
\end{equation}
It is clear that $u_n$ is increasing with $n$.

If $\sigma=1$, taking $u_n$ as a test function in \eqref{aprn} and
using the Hardy-Sobolev inequality it follows that
$$
\dfrac{a_{N,s}}{2}{}\left(1-\frac{\l}{\Lambda_{N,s}}\right)[u_n]_{\dot{H}^{s}(\ren)}^2\le C.
$$
Hence $\{u_n\}_{n\in\mathbb{N}}$ is uniformly bounded in $\dot{H}^{s}(\ren)$ and then, up to a subsequence,
$u_n\rightharpoonup u$ weakly in $X^s(\mathbb{R}^N)$, where $u$
solves \eqref{rn}. Using the monotonicity of $u_n$ and a straightforward adaptation of Lemma \ref{compa} we can
prove that $u_n\to u$ strongly in $\dot{H}^{s}(\ren)$, which proves (i), and (iii) similarly follows.

To prove (ii) we take $G_k(u_n)$ as a test
function in \eqref{aprn} and performing the same computations as in
the proof of Theorem \ref{m1} we conclude.

Finally, (iv) follows closely using the arguments in the
proof of Theorem \ref{m1}. In particular, the existence of $u\in L^1(\Omega)$ is a consequence of the uniform bounds
of $\{G_k(u_n)\}_{n\in\mathbb{N}}$ and $\{T_k^{\frac{\sigma+1}{2}}(u_n)\}_{n\in\mathbb{N}}$ in $\dot{H}^s(\R^N)$ and
the fact that $\frac{\sigma+1}{2}<1$.
\end{proof}

\begin{remark} In a similar way to the results in \cite{BMP} the problem
\begin{equation}\label{ictp}
\left\{
\begin{array}{rcll}
(-\Delta)^{s}u &= &\lambda \dfrac{u}{|x|^{2s}}+\mu\dfrac{h(x)}{u^{{\sigma}}}+u^p &{\rm in}\; \Omega,\\
u &> & 0 &{\rm in}\;\, \Omega, \\
u &= &0 &{\rm in}\,\,   \ren\setminus\Omega,
\end{array}
\right.
\end{equation}
can be analyzed. In fact the existence of a minimal solution for  $p<p_+(\lambda)$ and $\l\le \Lambda_{N,s}$ and $\mu$ small can be done with minor analytical changes. Indeed,

\begin{Theorem}\label{ictp2} Assume that in problem \eqref{ictp}  one of the following conditions holds:
\begin{enumerate}
\item[(i)] ${\sigma}<1$, $p<2^*_s-1$ and $h\in L^\theta(\Omega)$ for some $\theta\ge \frac{2^*_s}{2^*_s-(1-\sigma)}$.
\item[(ii)] $\sigma\ge 1$, $p<2^*_s-1$ and  $h\in L^{\frac{2^*_s}{2^*_s-(1-\sigma_1)}}(\Omega)$ for some $\sigma_1<1$.
\item[(iii)] $\sigma>0$, $2^*_s-1\le p<p_+(\lambda)$ and $h\in L^{\infty}(\Omega)$.
\end{enumerate}
Then there exists $\mu^*>0$ such that  for all $\mu<\mu^*$, problem \eqref{ictp} has a minimal weak solution $u$ and, moreover, for  all $\mu>\mu^*$, problem \eqref{ictp} has no positive solution.
\end{Theorem}
The existence of a second positive solution in the cases (i) and (ii) for $\mu$ small enough is easy to obtain. The result for all $\mu<\mu^*$  by a direct method seems to be an open problem.
\end{remark}

\end{document}